\documentclass[11pt, a4paper]{amsart}
\usepackage{bbm, amssymb, mathrsfs, comment}
\newtheorem{theorem}{Theorem}[section]
\newtheorem{lemma}[theorem]{Lemma}

\newtheorem{corollary}[theorem]{Corollary}
\newtheorem{proposition}[theorem]{Proposition}

\newtheorem{innercustomgeneric}{\customgenericname}
\providecommand{\customgenericname}{}
\newcommand{\newcustomtheorem}[2]{%
  \newenvironment{#1}[1]
  {%
   \renewcommand\customgenericname{#2}%
   \renewcommand\theinnercustomgeneric{##1}%
   \innercustomgeneric
  }
  {\endinnercustomgeneric}
}

\newcustomtheorem{customthm}{Theorem}
\newcustomtheorem{customprop}{Proposition}
\newcustomtheorem{customdef}{Definition}

\newtheorem{definition}[theorem]{Definition}

\newtheorem{remark}[theorem]{Remark}
\newtheorem{example}[theorem]{Example}

\newcommand{\I}{{\mathbbm{1}}}
\newcommand{\M}{\mathcal{M}}
\newcommand{\cH}{\mathcal{H}}
\newcommand{\cM}{\mathbb{M}}
\newcommand{\tcM}{\widetilde{\mathbb{M}}}

\begin{document}
\title[]{Quantum Fokker-Planck dynamics}

\author{Louis Labuschagne}

\address{DSI-NRF CoE in Math. and Stat. Sci,\\ Focus Area for Pure and Applied Analytics,\\ Internal Box 209, School of School of Math. $\&$ Stat. Sci.\\NWU, PVT. BAG X6001, 2520 Potchefstroom\\ South Africa}
\email{Louis.Labuschagne@nwu.ac.za}

\author{W Adam Majewski}

\address{DSI-NRF CoE in Math. and Stat. Sci,\\ Focus Area for Pure and Applied Analytics,\\ Internal Box 209, School of School of Math. $\&$ Stat. Sci.\\NWU, PVT. BAG X6001, 2520 Potchefstroom\\ South Africa}
\email{fizwam@gmail.com}

\keywords{Markov semigroups, noncommutative $L^p$-spaces, Fokker-Planck dynamics, Csiszar-Kullback inequality, von Neumann algebras}

\subjclass{Primary: 46L55, 47D07; Secondary: 46L51, 46N50, 46L57}
\date{\today}

\begin{abstract}
The Fokker-Planck equation is a partial differential equation which is a key ingredient in many models in physics. This paper aims to obtain a quantum counterpart of Fokker-Planck dynamics, as a means to describing quantum Fokker-Planck dynamics. 
Given that relevant models relate to the description of large systems, the quantization of the Fokker-Planck equation should be done in a manner that respects this fact, and is therefore carried out within the setting of non-commutative analysis based on general von Neumann algebras. Within this framework we present a quantization of the generalized Laplace operator, and then go on to incorporate a potential term conditioned to noncommutative analysis. In closing we then construct and examine the asymptotic behaviour of the corresponding Markov semigroups. We also present a noncommutative Csiszar-Kullback inequality formulated in terms of a notion of relative entropy, and show that for more general systems, good behaviour with respect to this notion of entropy ensures similar asymptotic behaviour of the relevant dynamics.
\end{abstract}
\maketitle
\bigskip

\section{Introduction}
The Fokker-Planck equation is a partial differential equation that describes the time evolution of some basic models in classical statistical mechanics. It is worth pointing out that one of the fundamental problems in the kinetic theory of matter is the analysis of convergence to equilibrium for a given time evolution. In particular, such analysis is using the time monotonicity of the physical entropy. A nice  example of such analysis is provided by the seminal paper \cite{AMTU}, where additionally the strategy based on carr{\'e} du champ was used. We recall that carr{\'e} du champ and its implications  play a central role in the analysis and geometry of Markov diffusion operators in the Bakry-Emery theory, see \cite{BGL}.
In particular, this concept is a useful way of characterizing the infinitesimal behaviour of a Markov semigroup.

However, the fundamental theory in Physics is quantum mechanics. Because of this fact some quantum Fokker-Planck models were studied recently, see \cite{AF}, \cite{ALMS}. The analysis of the above models was based on the phase-space formulation of quantum mechanics. It leads to Wigner functions as well as to the Wigner-Fokker-Planck equations. We note that the phase space formulation of quantum mechanics, although not equivalent to the original Heisenberg representation, has the advantage that it makes quantum mechanics appear  similar to the classical  Hamiltonian mechanics. In particular, the phase-space formulation examines of some differential equations similar to that in classical statistical mechanics.

However, it seems that large quantum systems (statistical mechanics with the thermodynamic limit, quantum field theory) are not in the scope of this formulation, see \cite{Y}.  Therefore, to study genuine quantum models we will consider the von Neumann algebra generated by the bounded observables as the primary object for a mathematical formalism of quantum theory.  One of the salient aspects of the algebraic quantum mechanical description is the non-commutative calculus. In other words, we have to use both non-commutative integration and non-commutative differential theory, cf \cite{Maj2}, to achieve a truly quantum theory. In particular, the Hilbert space $L^2(\mathcal{M})$ consisting of selected  quantum measurable operators will replace the  ``classical''  $L^2$- Hilbert space. The important point to note here is that to define carr{\'e} du champ one needs $^*$-subalgebras in this $L^2$-space. As Haagerup's $L^2(\mathcal{M})$ does not have this property, the Bakry-Emery strategy is not immediately extendable for genuine quantum systems in quantum statistical mechanics.

Fortunately, for such general systems, the relative entropy is well defined, see \cite{Ar}. In the presented analysis,  following  the idea of non-commutative perspectives,  see \cite{E}, \cite{EH}, we will use the relative operator entropy $S(A|B)$ cf \cite{FK1}, \cite{FK2}, \cite{Petz}. In other words, the prescription for quantum relative entropy will be slightly modified.  Furthermore, the above modification of quantum relative entropy together with the theory of quantum Markov semigroups seems to be best adapted to carry out the "quantization" of the analysis given in Arnold et al paper \cite{AMTU}.

But to this end one needs to quantize the Laplace operator and to describe the associated quantum Markov semigroups. In particular, quantum Dirichlet forms which are related to Markov semigroups should be examined. We recall that the need to quantize Dirichlet forms as well as Markovian semigroups were recognized a long time ago, both in Quantum Field Theory and Quantum Statistical Mechanics, see \cite{Cip} for a detailed account of these achievements.

However, as the basic objective of this paper is to put in place a line of attack for producing a quantum version of key parts of \cite{AMTU}, the quantization of generalised Laplace operators within the framework of Haagerup's $L^2$-spaces will be essential.

The idea of quantization of diffusive processes within a non-commutative $L^2$-space  goes back at least as far as to \cite{MZ1, MZ2}. However, the construction of quantum diffusive dynamics, which appeared in \cite{MZ1, MZ2}, was done for lattice systems under very strong abelianess conditions. This approach was amended by Park \cite{P1,P2}. It was carried out on the standard forms of von Neumann algebras. However, Park's quantization was based on maps which only by certain abuse of language can be called derivatives.

As the Fokker-Planck equation is designed for general continuous systems, it is natural to begin with a general form of non-commutative derivations. In particular, this explains why we will not follow the \cite{CS} strategy where the first order differential calculus is associated to regular Dirichlet forms on $C^*$-algebras.

But, as our primary goal is quantization and analysis of quantum Fokker-Planck equation, we will select those derivations which are compatible with dynamical maps, see Theorem 3.2.50 in \cite{BR}. This will be done in Section \ref{Lapdyn}. Furthermore, to include a potential, as it was done in Fokker-Planck equation, quantum $L^2$-Markov semigroups corresponding to selected potentials will be examined and described.

The paper is organized as follows: In Section \ref{s2} we recall the essentials of non-commutative $L^p$-integration theory. Then in section \ref{markov} we pass to an analysis of the general theory of $L^p$-Markov semigroups. Section \ref{PSP} then builds on that theory to demonstrate the existence of Markov semigroups induced by some selected potentials. In Section \ref{Lapdyn} we proceed with the study of quantum Laplacian Markov semigroups. In particular, we show how the derivations comprising the quantum Laplacian may be selected to ensure that the Laplacian leads to well defined Markov semigroups on $L^2$-space.
With all the underlying theory now in place, Section \ref{FPD} then provides a detailed exposition of Quantum Fokker-Planck dynamics. Here Theorem \ref{FPgen} is the key result describing Quantum Fokker-Planck dynamics.
As was mentioned at the beginning of this section, convergence to equilibrium is a fundamental problem in the kinetic theory. In recognition of this fact, Section \ref{asymp} is devoted to the study of the asymptotic behaviour of the considered dynamics. We then pass to the challenge of finding conditions which guarantee similar asymptotic behaviour in more general systems. This is achieved by firstly introducing the notion of quantum Csiszar-Kullback inequality in Section 
\ref{qcki}, which is then used in Section \ref{genconv} to describe situations where we still get convergence to equilibrium.  

Finally, we want to clarify the relation of the proposed quantization of Fokker-Planck dynamics to Quantum Physics. To begin with, we recall that in Dirac's program for the mathematical formulation of the fundamentals of Quantum Mechanics (see the opening chapters in \cite{Dir}), the time evolution should be based on differentiation leading to differential equations. In particular, when this ideology of Dirac is considered alongside some of the techniques and advances from the theory of operator algebras, one can see that when applied to time evolving observables, differentiation leads derivations of the system. For details see Section 1 of \cite{KadLiu}. Consequently, derivations constitute a milestone in the analysis of the dynamics of quantum systems.

This approach was applied to a description of the dynamics of (infinite) quantum spin systems - see Section 6.2 in \cite{BR} and also \cite{Araki} and references given there for a more recent account .

Turning to (infinite) quantum continuous systems, free Fermi or Bose gas, we note that Park and his school developed an approach based on quantum Dirichlet forms to study dynamics for such systems, see \cite{BKP1}, \cite{BKP2} and/or \cite{Cip}. However, as was mentioned earlier in Introduction, these forms do not have simple connections to derivatives. Therefore bearing the successful applications of derivatives to the description of quantum spin systems in mind, the quantization of Fokker-Planck dynamics begins with derivatives. Subsequently, the relations among Dirichlet forms and Markovianity are employed. Consequently, it can be said that our approach is faithful to the quantization procedure of dynamical systems.

\section{Preliminaries}\label{s2}

Throughout this paper $\M$ will be a $\sigma$-finite von Neumann algebra equipped with a faithful normal state $\nu$. For the construction of quantum $L^p$-spaces associated with such an algebra, we shall pass to the crossed product $\cM=\M\rtimes_\nu\mathbb{R}$. Recall that $\cM$ is semifinite and that it admits a dual action $\theta_s$ 
($s\in\mathbb{R}$) of the reals. The crossed product also admits a canonical faithful normal semifinite trace $\tau$ characterised by $\tau\circ \theta_s = e^{-s}\tau$. So $\cM$ may be enlarged to the algebra of $\tau$-measurable operators $\tcM$, with $L^p(\M)$ defined by $L^p(\M)=\{a\in\tcM\colon \theta_s(a)=e^{-s/p}a\mbox{ for all }s\in\mathbb{R}\}$. It is known that $L^\infty(\M)$ is a *-isomorphic copy of $\M$. The space $L^1(\M)$ admits a trace functional $tr$ which can be used to develop a theory of $L^p$-duality. For the case $1\leq p<\infty$, the $L^p$-norm is given by $\|a\|_p=tr(|a|^p)^{1/p}$.

There further exists a so called ``operator valued weight'' $\mathscr{W}$ from $\cM_+$ onto the extended positive part of $\M$ 
(see \cite[IX.4.4 \& IX.4.12]{Tak2} for these concepts). For any normal semifinite weight $\nu$ on $\M$, 
$\widetilde{\nu}=\nu\circ \mathscr{W}$ will be a normal semifinite weight on $\cM$ - the so-called dual weight of $\nu$. By 
the theory in \cite{PT}, this dual weight is of the form $\tau(h_\nu^{1/2}\cdot h_\nu^{1/2})$ where $\tau$ is the 
canonical trace on the crossed product and where $h_\nu = \frac{d\widetilde{\nu}}{d\tau}$. In the case where $\nu$ is 
a normal state, $\nu$ can now be written as $\omega(a)=tr(h_\nu^{1/2}a h_\nu^{1/2})$, where 
$h_\nu\in L^1(\M)$. In fact on a similar note, for the density $h_\nu$, the prescription $a\to h_\nu^{c/p}a 
h_\nu^{(1-c)/p}$ will for each $0\leq c\leq 1$ define a dense embedding of $\M$ into $L^p(\M)$. In the sequel $h_\nu$ will where convenient simply be denoted by $h$.

We briefly show how the above structure compares to the GNS construction. In this construction one defines an inner product on $\M$ by means of the prescription $\langle a, b\rangle_{\nu}=\nu(b^*a)$. Under this inner product $\M$ becomes a pre-Hilbert space. One then writes $H_\nu$ for the completion of this pre-Hilbert space. To distinguish between $a\in \M$ considered as an element of $\M$ and as an element of $H_\nu$, we follow convention by writing $\eta(a)$ for $a$ considered as an element of $H_\nu$. Then $a\to \eta(a)$ becomes a dense embedding of $\M$ into $H_\nu$. We similarly know that in the above setting, $\{ah^{1/2}\colon a\in \M\}$ is a dense subspace of $L^2(\M)$ \cite[Theorem 7.45]{GLnotes}. Now observe that for any $a,b\in \M$, we have that $\langle a,b\rangle_{\nu} =\nu(b^*a)=tr(h^{1/2}b^*ah^{1/2})=tr((bh^{1/2})^*(ah^{1/2}))=\langle(ah^{1/2}),(bh^{1/2})\rangle$ where the second inner product is the canonical inner product on $L^2(\M)$. Thus the space $L^2(\M)$ is clearly a copy of $H_\nu$ with the densely defined map $\eta(a)\to ah^{1/2}$ extending to a linear isometry. We shall therefore where convenient freely pass from the GNS to the Haagerup setting using the identification $\eta(a)\leftrightarrow ah^{1/2}$

For a recent account of the theory of noncommutative $L^p$-spaces, we refer the reader to \cite{GLnotes}.

\section{$L^p$-Markov semigroups}\label{markov}

We start our investigation with a description of the general theory of (sub)$L^p$-Markov semigroups on Haagerup $L^p$-spaces associated with $\sigma$-finite von Neumann algebras. The fountainhead of this theory is a simple but elegant Radon-Nikodym type result by Schmitt (see Lemma \ref{RNLp}). This is of course not the first time these issues have been addressed. The theory we present in some sense represents an asymmetric version of section 2 of \cite{GL1}. The basic cycle of ideas underlying the theory were worked out in \cite{GL1, GL2}, with \cite[Proposition 5.4]{HJX} filling some gaps in that cycle and broaching the asymmetric theory. The presentation hereafter therefore follows familiar patterns with suitable modifications being made to facilitate the passage to the asymmetric context.

\begin{definition}
Let $p\in [1,\infty ]$ and let $T$ be a densely defined operator on $L^p(\M)$. Then
\begin{itemize}
\item $T$ is \textit{reality preserving} if for every $x\in\mathrm{dom}(T)$ we have that $x^{*}\in \mathrm{dom}(T)$ and $T(x^{*})=(Tx)^{*}$;
\item $T$ is \textit{positivity preserving} if (i) $\mathrm{dom}(T)=\mathrm{span}(\mathrm{dom}(T)^{+})$ 
and (ii) $Tx\geq 0$ for all $x\in \mathrm{dom}(T)_{+}$ (where $\mathrm{dom}(T)^{+}=\mathrm{dom}(T)\cap L^p(\M)^+$);
\item If $T$ is bounded and defined on all of $L^p$ (where $p<\infty$), we say it is ($L^p$-)\textit{sub-Markov} if whenever $0\leq x\leq h^{1/p}$  holds, we have that $0\leq T(x)\leq h^{1/p}$. If additionally $T(h_{\nu}^{1/p}) = h_{\nu}^{1/p}$ we say that it is Markov.
\end{itemize}
\end{definition}

\begin{definition}
A bounded $T$ on $\mathcal{M}$ is said to be $p$-\emph{integrable} (with respect to $\nu$) if the operator $T_0^{(p)}$ 
defined by $T_0^{(p)}\colon  h^{1/2p}ah^{1/2p} \mapsto h^{1/2p}T(a)h^{1/2p}$ is $L^p$-continuous. If $T$ is $p$-integrable 
we write $T^{(p)}$ for the unique continuous extension of $T_0^{(p)}$ to $L^p(\mathcal{M})$. A 1-integrable operator will 
simply be said to be \emph{integrable}. 
\end{definition}

\begin{definition}
A pair of bounded operators $(T,S)$ on $\M$ is said to be a \emph{KMS pair} (with respect to $\nu$) if $tr(Ta(h^{1/2}bh^{1/2}))=tr((h^{1/2}ah^{1/2})Sb)$ for all $a,\,b\in M$.
The operator $S$ may alternatively also be referred to as the $KMS$-adjoint of $T$. We say that $T$ is \emph{KMS-symmetric} if 
$(T,T)$ is a \emph{KMS pair}.
\end{definition}

\begin{definition}\label{DBIIdef}
A positive map $T$ on $\M$ is said to satisfy Detailed Balance II (often abbreviated to DBII) with respect to the faithful normal state $\nu$ if there exists another positive map $T^\flat$ on $\M$ such that $\nu(T(a)b)=\nu(aT^\flat(b))$ for all $a,b\in \M$.
\end{definition}

\begin{remark} There is actually a close relation between the concept of a KMS adjoint and DBII. Note that the the equality 
$\nu(T(a)b)=\nu(aT^\flat(b))$ used in the definition of DBII, may also be written as $$tr((h^{1/2}T(a))(bh^{1/2}))= tr((h^{1/2}a)(T^\flat(b)h^{1/2})),$$whilst the criterion used in the definition of KMS-adjoints, may be written as 
$$tr((h^{1/4}Ta(h^{1/4})(h^{1/4}bh^{1/4}))=tr((h^{1/4}ah^{1/4})(h^{1/4}Sbh^{1/4})).$$Recalling that for $0\leq c\leq 1$, $h^{c/2}\M h^{(1-c)/2}$ is a dense subspace of $L^2(\M)$, both definitions make a statement about how a pair of operators behave with respect to each other when lifted to densely defined operators on $L^2(\M)$ by means of a prescription of the form $h^{c/2}ah^{(1-c)/2} \to h^{c/2}T(a)h^{(1-c)/2}$. The two definitions just use different embeddings of $\M$ into $L^2(\M)$. \end{remark}

We observe

\begin{remark}\label{natcone} If $p=2$ then for the dense subset $\Delta^{1/4} \mathcal{M}_+ h^{1/2}_{\nu}$ of the natural 
cone one has $$\Delta^{1/4} \mathcal{M}_+ h^{1/2}_{\nu} = h_{\nu}^{1/4} \mathcal{M}_+ h^{1/4}_{\nu}.$$
Thus, in the Haagerup setting $L^2_+ (\mathcal{M})$ can be considered as the natural cone.
\end{remark}

\begin{lemma}[Lemma 2.2, \cite{Sch}]\label{RNLp}
If $x\in L^p(\M)$ $(1\leq p\leq \infty$) satisfies $0\leq x\leq h^{1/p}$, then there exists $a\in \M_+$ such that $x=h^{1/2p}ah^{1/2p}$.
\end{lemma}

The preceding Radon-Nikodym type result, is the key tool in the proof of the following lemma, which shows that all $L^p$-Markov maps are induced by Markov maps on the underlying von Neumann algebra.

\begin{lemma}[Proposition 2.5,\cite{GL1}]\label{LpMarkov}
Let $\mathcal{T}$ be an $L^p$-Markov operator on $L^p(\M)$ ($1\leq p\leq \infty$). Then there exists a Markov operator $T$ on $\M$ such that $h^{1/2p}T(a)h^{1/2p}=\mathcal{T}(h^{1/2p}ah^{1/2p})$ for all $a\in \M$.
\end{lemma}

\begin{proof} For any $a\in \M$ with $0\leq a\leq\I$, we have that $0\leq h^{1/2p}ah^{1/2p}\leq h^{1/p}$ and hence that $0\leq \mathcal{T}(h^{1/2p}ah^{1/2p})\leq h^{1/p}$ since $\mathcal{T}$ is $L^p$-Markov. So there exists $a_T$ such that $h^{1/2p}a_Th^{1/2p}=\mathcal{T}(h^{1/2p}ah^{1/2p})$. The fact that $0\leq \mathcal{T}(h^{1/2p}ah^{1/2p})\leq h^{1/p}$, ensures that $0\leq a_T\leq \I$. Then next task is to show that $a\to a_T$ is an affine map on $\M_+$. This map then extends to a linear map $T$ on $\M$ which preserves the positive cone $\M_+$ and for which we have that $0\leq T(a)\leq \|a\|\I$ for all $a\in\M_+$. So $T$ is a Markov map.
\end{proof}

The next two lemmata and the theorem that follows capture the essential information encoded in \cite[Proposition 5.4]{HJX}. The first lemma is a ``KMS-version'' of the corresponding fact for positive maps satisfying DBII with the two proofs following similar paths. Its value is in providing sufficient conditions for integrability of Markov operators. It represents an asymmetric version of \cite[Proposition 4.5]{GL2}.

\begin{lemma}\label{intmar}
Let $\mathcal{T}$ be an $L^2$-Markov operator for which $\mathcal{T}^*$ is also $L^2$-Markov. Then the operator $T$ described in the preceding lemma satisfies $\nu\circ T\leq\nu$.
\end{lemma}

\begin{proof} For any $a\in \M_+$ we have that 
$$\nu(T(a))=tr(h^{1/2}T(a)h^{1/2})= tr(h^{1/2}\mathcal{T}(h^{1/4}ah^{1/4}))= tr(\mathcal{T}^*(h^{1/2})(h^{1/4}ah^{1/4}))$$ 
$$\leq tr(h^{1/2}(h^{1/4}ah^{1/4})) = tr(h^{1/2}ah^{1/2})=\nu(a).$$
\end{proof}

Our next lemma is an asymmetric version of \cite[Proposition 2.4]{GL1} for the case $p=1$.
 
\begin{lemma}
Let $\mathcal{T}$ be an $L^2$-Markov operator for which $\mathcal{T}^*$ is also $L^2$-Markov, and let $T$ and $S$ be the 
Markov operators on $\M$ corresponding to $\mathcal{T}$ and $\mathcal{T}^*$ respectively. Then the prescription 
$h^{1/2}ah^{1/2}\to h^{1/2}S(a)h^{1/2}$ ($a\in \M$) extends to a contractive map $S^{(1)}$ on $L^1(\M)$ for which the Banach 
dual is $T$. 
\end{lemma}

\begin{proof} The first claim is an easy consequence of \cite[Theorem 5.1]{HJX}. Now observe that for any $a,b\in \M$ we have 
that 
\begin{eqnarray*}
tr(S^{(1)}(h^{1/2}ah^{1/2})b)&=& tr((h^{1/2}S(a)h^{1/2})b)\\
&=& tr(h^{1/4}\mathcal{T}^*(h^{1/4}ah^{1/4}))h^{1/4}b)\\
&=& tr(\mathcal{T}^*(h^{1/4}a^*h^{1/4}))^*(h^{1/4}bh^{1/4}))\\
&=& \langle(h^{1/4}bh^{1/4}),\mathcal{T}^*(h^{1/4}a^*h^{1/4}))\rangle \\
&=& \langle(\mathcal{T}(h^{1/4}bh^{1/4}),(h^{1/4}a^*h^{1/4}))\rangle \\
&=& \langle(h^{1/4}T(b)h^{1/4}),(h^{1/4}a^*h^{1/4})\rangle \\
&=& tr((h^{1/4}ah^{1/4})(h^{1/4}T(b)h^{1/4})) \\
&=& tr((h^{1/2}ah^{1/2})T(b)).
\end{eqnarray*}
By continuity and the density of $h^{1/2}\M h^{1/2}$ in $L^1(\M)$, we have that  $tr(S^{(1)}(f)b)=tr(fT(b))$ 
for all $f\in L^1(\M)$.
\end{proof}

We are now finally ready to present the main result of this section, which establishes a connection between $L^2$ and weak* $L^\infty$ Markov semigroups. This result is in principle a KMS-assymetric version of \cite[Theorem 2.6]{GL1} for semigroups rather than for single operators.

\begin{theorem}
Let $(\mathcal{T}_t)$ be an $L^2$-Markov semigroup for which $(\mathcal{T}_t^*)$ is also an $L^2$-Markov semigroup, and let $(T_t)$ be the set of Markov operators on $\M$ corresponding to $(\mathcal{T}_t)$. Then each $T_t$ is normal and $(T_t)$ is a weak* semigroup. 
\end{theorem}

\begin{proof} The normality of the $T_t$'s follow from the preceding lemma. The semigroup property is easy to check, and so all that remains is to show that $T_t(a)$ is weak*-convergent to $a$ as $t\to 0$. For any $a,b\in \M$ we have that 
\begin{eqnarray*}
\lim_{t\to 0}tr(T_t(a)(h^{1/2}bh^{1/2}))&=& \lim_{t\to 0}tr((h^{1/4}T_t(a)h^{1/4})(h^{1/4}bh^{1/4}))\\
&=& \lim_{t\to 0} tr(\mathcal{T}_t(h^{1/4}ah^{1/4}))(h^{1/4}bh^{1/4}))\\
&=& tr((h^{1/4}ah^{1/4}))(h^{1/4}bh^{1/4})) \\
&=& tr(a(h^{1/2}bh^{1/2})).
\end{eqnarray*}
We may now use that facts that $\|T_t\|\leq 1$ for each $t$, and that $h^{1/2}\M h^{1/2}$ is dense in $L^1(\M)$ to conclude that 
$\lim_{t\to 0}tr(T_t(a)f)=tr(af)$ for each $f\in L^1(\M)$. The result therefore follows.
\end{proof}

In closing we provide conditions under which the $L^\infty$-semigroups described above satisfy DBII.
 
\begin{theorem}\label{markovdb2}
Let $(\mathcal{T}_t)$ be an $L^2$-Markov semigroup for which $(\mathcal{T}_t^*)$ is also an $L^2$-Markov semigroup, and let $(T_t)$ and $(T_t^\flat)$ be the weak* Markov semigroups on $\M$ corresponding to $(\mathcal{T}_t)$ and $(\mathcal{T}_t^*)$ respectively. Then $(T_t)$ and $(T_t^\flat)$ are KMS-adjoints of each other. If $(\mathcal{T}_t)$ commutes with the modular operator $\Delta^{1/4}$, then $(T_t)$ and $(T_t^\flat)$ satisfy DB II with respect to each other.
\end{theorem}

\begin{proof} The proofs of the two claims are very similar, and hence we will only prove the second claim. If $[\mathcal{T}_t,\Delta^{1/4}]=0$, then also $[\mathcal{T}_t^*,\Delta^{1/4}]= -[\mathcal{T}_t,\Delta^{1/4}]^* 
=0$. But in that case we have that $T_t(a)h^{1/2}= \Delta^{-1/4}(h^{1/4}T_t(a)h^{1/4}) = 
\Delta^{-1/4}\mathcal{T}_t(h^{1/4}ah^{1/4}) = \mathcal{T}_t(\Delta^{-1/4}(h^{1/4}ah^{1/4})) = \mathcal{T}_t(ah^{1/2})$ for any 
$a\in \M$ and any $t$. We similarly have that $T_t^\flat(a)h^{1/2}=  \mathcal{T}_t^*(ah^{1/2})$ for any $a\in \M$ and any $t$. 
But in that case $\nu(bT_t(a))= tr(h^{1/2}bT_t(a)h^{1/2}) = tr(h^{1/2}b\mathcal{T}_t(ah^{1/2})) = 
\langle\mathcal{T}_t(ah^{1/2}), b^*h^{1/2}\rangle =\langle ah^{1/2},\mathcal{T}_t^*(b^*h^{1/2})\rangle = 
\langle ah^{1/2},T_t^\flat(b^*)h^{1/2})\rangle = \nu(T_t^\flat(b)a)$.
\end{proof}

\section{Potential dynamics - positive maps as generators of strongly positive semigroups}\label{PSP}

The classical Fokker-Planck equation can be written in form comprising a ``Potential'' and ``generalised Laplacian'' term (see equation (2.4) in \cite{AMTU}). We shall propose quantum counterparts of each of these terms, starting with the potential term in this section. In our investigation we will restrict ourselves to von Neumann algebras $\M$ equipped with a faithful normal state $\nu$. Throughout this section we will assume that $T$ is a unital Schwarz map. For $V=T-\mathrm{id}$ we will for any $a\in \M$ have that 
\begin{eqnarray*}
V(a^*a)&=&T(a^*a)-a^*a\\
&\geq& T(a)^*T(a)-a^*a\\
&\geq& T(a)^*T(a)-a^*a-|T(a)-a|^2\\
&=&T(a)^*a+a^*T(a)-2a^*a\\
&=&V(a)^*a+a^*V(a).
\end{eqnarray*} 
This inequality leads to the following conclusion:

\begin{proposition} For any unital Schwarz map, $V=T-\mathrm{id}$ generates a uniformly continuous semigroup $(P_t)$ of unital strongly positive maps on $\M$.
\end{proposition}

This follows from the Evans - Hanche-Olsen result (see the Notes and Remarks regarding \S 5.3.1 in \cite{BR2}). The unitality of the semigroup $(P_t)$ is a consequence of the fact that $V(\I)=0$. Now suppose that in addition $\nu\circ T\leq \nu$. Since the semigroup is uniformly continuous, each $P_t$ is of the form $P_t=e^{tV}$. But then $\nu\circ P_t=e^{-t}\nu\circ e^{tT}=e^{-t}\sum_{k=0}\frac{1}{k!}t^k\nu\circ T^k\leq e^{-t}\sum_{k=0}\frac{1}{k!}t^k\nu=\nu$. So by \cite[Theorem 5.1]{HJX}, the action of the semigroup will in this case extend to each $L^p(\M)$ by means of the prescription $P_t^{(p)}(h^{1/2p}ah^{1/2p})=h^{1/2p}P_t(a)h^{1/2p}$ ($a\in \M$). As far as $L^2(\M)$ is concerned, we then also have that $\|P_t(a)h^{1/2}\|_2^2=tr(h^{1/2}P_t(a)^*P_t(a)h^{1/2})\leq tr(h^{1/2}P_t(a^*a)h^{1/2})=\|P_t^{(1)}(h^{1/2}a^*ah^{1/2})\|_1\leq \|h^{1/2}a^*ah^{1/2}\|_1=\|ah^{1/2}\|_2^2$ for all $a\in \M$. So $(P_t)$ extends to a uniformly continuous semigroup $(\widetilde{P}_t)$ on $L^2(\M)$ for which the prescription $t\to h^{1/2}\widetilde{P}_t$ yields a densely defined semigroup on $L^1(\M)$ which extends by continuity to $(P_t^{(1)})$. We remind the reader that the embedding $\M\to L^2(\M):a\mapsto ah^{1/2}$ corresponds to the construction of $H_\nu$ by means of the GNS-construction. So up to this equivalence the semigroup $(\widetilde{P}_t)$ is just a semigroup on $H_\nu$ constructed by means of the prescription $\widetilde{P}_t(\eta(a))=\eta(P_t(a))$ for all $a\in \M$. The one problem we face here is that since the semigroup $(\widetilde{P}_t)$ is not defined by means of a symmetric embedding, it is not a priori clear that this semigroup is $L^2$-Markov. However with some additional assumptions this problem can be circumvented.

The following theorem shows that under mild assumptions, the semigroup induced by $V$ is of exactly the type we need. Operators of this form are therefore a rich source of ``quantum potential dynamics''.
 
\begin{theorem}\label{Vcommdelta} Let $T$ be a unital positive map on $\M$ which satisfies DBII with respect to $\nu$, and let $V=T-\mathrm{id}$ be as before. Then the prescription $P_t^{(1)}(h^{1/2}ah^{1/2})=h^{1/2}P_t(a)h^{1/2}$ ($a\in \M$) extends continuously to an $L^1$-Markov semigroup on $L^1(\M)$, whilst the prescription $\widetilde{P}_t(ah^{1/2})= P_t(a)h^{1/2}$ ($a\in \M$) extends continuously to an $L^2$-Markov semigroup on $L^2(\M)$ for which $(P_t^{(1)})$ appears as the continuous extension of the partially defined semigroup $(h^{1/2}\widetilde{P}_t)$ on $L^1(\M)$. The operator $V$ yields a bounded operator $\widetilde{V}$ on $L^2(\M)\equiv H_\nu$ by means of the prescription $\widetilde{V}(\eta(a))=\eta(V(a))$ ($a\in \M$) which turns out to be the infinitesimal generator of $(\widetilde{P}_t)$. In addition $\widetilde{V}$ commutes with $\Delta^{1/4}$, and each $\widetilde{P}_t$ is then also induced by the prescription $h^{1/4}ah^{1/4}\to h^{1/4}P_t(a)h^{1/4}$ ($a\in \M$).
\end{theorem}

\begin{proof} If $T$ satisfies DBII, then $\nu\circ T=\nu$, and hence most of the claims follow from the preceding discussion. 
What remains to be done is to show that $(\widetilde{P}_t)$ is $L^2$-Markov with bounded generator $\widetilde{V}$, that 
$\widetilde{V}$ commutes with $\Delta^{1/2}$, and that each $\widetilde{P}_t$ is an extension of $h^{1/4}ah^{1/4}\to h^{1/4}P_t(a)h^{1/4}$ ($a\in \M$). As far as the boundedness of 
$\widetilde{V}$ is concerned, notice that for $T$ we have that $\|\widetilde{T}(\eta(a))\|_2^2=\|\eta(T(a))\|_2^2= \nu(T(a)^*T(a))\leq \nu(T(a^*a))=\nu(a^*a)=\|\eta(a)\|_2^2$ for all $a\in \M$. Let $T^\flat$ be the unital positive map for 
which we have that $\nu(aT(b))=\nu(T^\flat(a)b)$ for all $a,b\in \M$, and write $V^\flat$ for $T^\flat-\mathrm{id}$ and
$P_t^\flat$ for $e^{tV^\flat}$. The same analysis as was applied to $V$, also shows that $(P_t^\flat)$ is a strongly positive 
unital semigroup. For any $a,b\in \M$ we will then have that $\nu(aV(b))=\nu(a(T(b)-b))=\nu((T^\flat(a)-a)b)= \nu(V^\flat(a)b)$, and hence that $\nu(aP_t(b))=\sum_{k=0}^\infty\frac{1}{k!}\nu(aV^k(b))= \sum_{k=0}^\infty\frac{1}{k!}\nu((V^\flat)^k(a)b)= \nu(P_t^\flat(a)b)$. Thus the pair $(P_t,P_t^\flat)$ satisfy DBII with 
respect to each other. This in particular ensures that for each $t$, the induced operator $\widetilde{P}_t$ on $H_\nu$, 
commutes with the modular operator \cite[Corollary 4]{MS}. We already know that the positivity of $P_t$ ensures that each 
$\widetilde{P}_t$ maps $\eta(\M_+)$ back into itself. When these two facts are combined, the continuity of $\widetilde{P}_t$ 
ensures that we may conclude from \cite[Proposition 2.5.26]{BR} that each $\widetilde{P}_t$ preserves that natural cone. That 
is each $\widetilde{P}_t$ preserves the order structure on $H_\nu$. Since also each $\widetilde{P}_t$ preserves 
$\eta(\I)$, this semigroup is $L^2$-Markov. On passing to the Haagerup-Terp context the commutation of the $\widetilde{P}_t$'s with $\Delta^{1/4}$ also ensure that $h^{1/4}P_t(a)h^{1/4}=\Delta^{1/4}(P_t(a)h^{1/2})= \Delta^{1/4}\widetilde{P}_t(ah^{1/2})= \widetilde{P}_t(\Delta^{1/4}(ah^{1/2}))= \widetilde{P}_t(h^{1/4}ah^{1/4})$ for each $t$.

To see that $\widetilde{V}$ is the generator of $(\widetilde{P}_t)$, it is enough to note that we will for any $a\in \M$ have that 
$$\lim_{t\searrow 0}\frac{1}{t}(\widetilde{P}_t(\eta(a))-\eta(a))= \lim_{t\searrow 0}\frac{1}{t}\eta(\widetilde{P}_t(a)-a)  =\eta(V(a))=\widetilde{V}(\eta(a)).$$The fact that $\widetilde{V}$ commutes with $\Delta^{1/4}$ easily follows from the known theory of positive maps satisfying DBII \cite{MS}.
\end{proof} 

There is one further refinement that can be made if additionally we assume the map $T$ to be CP. We pause to note this fact.

\begin{theorem} Let $T$ be a CP map for which we have that $\nu\circ T\leq \nu$ and also that $\sigma_t^\nu\circ T= T\circ\sigma_t^\nu$ for all $t$. For any $0\leq c\leq 1$ each of the prescriptions $\widetilde{T}_c(h^{c/2}ah^{(1-c)/2})= h^{c/2}T(a)h^{(1-c)/2}$ ($a\in \M$) will then yield the same contractive map on $L^2(\M)$.
\end{theorem} 

\begin{proof} By Theorem 4.11 of \cite{LM}, $T$ extends to a bounded map $\mathscr{T}$ on $(L^1+L^\infty)(\M\rtimes_\nu \mathbb{R},\tau)$. The same theorem shows that this space includes $L^2(\M)$, and that $\mathscr{T}$ restricts to a bounded 
map on $L^2(\M)$. The proof of \cite[Corollary 4.14]{LM} can now be modified to show that for any $0\leq c\leq 1$ and any 
$a\in \M$, we have that $\mathscr{T}(h^{c/2}ah^{(1-c)/2})= h^{c/2}T(a)h^{(1-c)/2}=\widetilde{T}_c(h^{c/2}ah^{(1-c)/2})$. This 
then proves the claim.
\end{proof}

\begin{remark} For our purposes we of course only need the above result to hold for the cases $c=0, \frac{1}{2}$. We pause to 
note that even then, the above result may not hold if the map $T$ does not commute with the modular automorphism group. To 
see this let $T$ be a positive map and define the maps $T_{(c)}$ by the prescription $T_{(c)}:h^{c/2} a h^{(1-c)/2} \to h^{c/2} T(a) h^{(1-c)/2}$ for all $a\in \M$. Suppose now that for both of the cases $c=0, \frac{1}{2}$ the map $T_{(c)}$ 
extends to the same bounded map $\widetilde{T}$ on $L^2(\M)$. Then $T_{(0)} V_0 \subset \overline{V_0}$ and 
$T_{(1/2)} V_{1/2} \subset \overline{V_{1/2}}$, where $V_{\alpha/2} = h^{\alpha/2}\M_+ h^{(1-\alpha)/2}$. Here 
$\overline{V_{1/4}}$ is of course the natural cone (see Remark \ref{natcone}). It then follows from \cite[Lemma 2]{BR3} that 
$\widetilde{T}$ must then commute with the modular operator if $\widetilde{T}^*$ also satisfies this property. In fact the 
remark immediately following \cite[Lemma 2]{BR3}, clearly shows why in general bounded extensions of $T_0$ and $T_{1/2}$ to 
$L^2(\M)$ will be different if this commutation criterion is not satisfied.
\end{remark}

\section{Laplacian dynamics}\label{Lapdyn}

We now pass to an identification of what may be termed a generalised quantum Laplacian. Let $\mathcal{F}=\{\delta_k\colon 1\leq k\leq n\}$ be a set of $n$ weak* closed derivations. We shall further assume that each derivation is unital in the sense that $\I\in \mathrm{dom}(\delta_i)$ for each $1\leq i\leq n$. In the case where the derivations are norm closed and norm densely defined this requirement is in fact automatically satisfied! (See Step 2 on page 22 of \cite{Bra}.) Hence it is an entirely natural assumption to make. For such sets of derivations we may follow Bratteli \cite{Bra} and introduce the sets of elements $\M^{(m)}(\mathcal{F})$ which are ``$m$-differentiable'' with respect to $\mathcal{F}$. These spaces are defined to be $$\M^{(m)}(\mathcal{F})=\{a\in\M\colon a\in \mathrm{dom}(\delta_{i_1}\ldots\delta_{i_m})\mbox{ for all }\delta_{i_1},\dots,\delta_{i_m}\in \mathcal{F}\}$$where $\M^{(0)}(\mathcal{F})$ is identified with $\M$. All in fact turn out to be Banach *-algebras when equipped with the norm 
$$\Vert a \Vert_{(n)} = \Vert a \Vert + \sum^m_{k=1} \frac{1}{k!}
\sum^{n}_{i_1=1} \ldots \sum^{n}_{i_k=1} \Vert\delta_{i_1} \delta_{i_2} \ldots \delta_{i_k}(a) \Vert.$$See for example \cite[2.2.4]{Bra} for a proof of these facts. We take note of the fact that the book \cite{Bra} focuses on $C^*$-algebras and norm-closed densely defined derivations. However when passing to the von Neumann algebra setting, norm-closedness must be replaced with weak* closedness, and densely defined with weak* densely defined. All relevant arguments from the $C^*$ context will then readily carry over to the von Neumann context. 

We shall further restrict ourselves to the setting where ${\M}^{(2)}(\mathcal{F})$ is weak* dense. Note that requiring each $\delta_i$ to be unital, ensures that ${\M}^{(2)}(\mathcal{F})$ is unital.

We next need find a means of extending these derivations to $H_\nu$. Let $\delta$ be a weak* densely defined weak*-closed *-derivation on $\M$. We then simply define an action of $\delta$ on the GNS Hilbert space $H_\nu$ by setting $\widetilde{\delta}(\eta(x))=\eta(\delta(x))$ for ``appropriate'' elements of $\M$, and $\eta$ the canonical map from $\M$ to $H_\nu$. 

The analysis of Cipriani and Sauvageot \cite[page 79]{CS} suggest that for a closed densely defined derivation $\delta$, one should regard $\delta^*\delta$ as the divergence of that derivation. Lending further credence to this idea is the observation that for smooth functions of compact support on the real line, integration by parts shows that 
$$\langle f', g\rangle =\int_{-\infty}^\infty f'\overline{g} = - \int_{-\infty}^\infty f\overline{g'} = -\langle f, g'\rangle.$$ So at least for these functions it seems that for the differential operator $D_x:f\to f'$, one has that $-D_x^2 f  = D_x^*D_x f$. Now let $D=[d_{ij}]$ be a positive definite numerical matrix. 

In the setting of ${\M}^{(2)}(\mathcal{F})$, the action of the extensions $\widetilde{\delta}_i$ are defined on $\eta({\M}^{(2)})$. The weak* density of ${\M}^{(2)}(\mathcal{F})$, then ensures that the common domain of these extensions - namely $\eta({\M}^{(2)})$ - is dense in $H_\nu$. So the expression $-\sum_{1\leq i,j\leq n} d_{ij}\widetilde{\delta}_i^*\widetilde{\delta}_j$ 
would then be a natural first guess for a quantum analogue of the generalised Laplacian $\mathrm{div}(D\nabla)$. However we want our Laplacian to be closed, and hence further analysis is required. We pause to note that the map $\sum_{1\leq i,j\leq n} d_{ij}\widetilde{\delta}_i^*\widetilde{\delta}_j$ is positive definite definite on its domain in the sense that $\langle\sum_{1\leq i,j\leq n} d_{ij}\widetilde{\delta}_i^*\widetilde{\delta}_j(\xi),\xi\rangle\geq 0$ for all $\xi$ in its domain. This can be seen as follows. Firstly regard the direct sum $\oplus_{i=1}^nH_\nu$ as a column space. The operator $$[\widetilde{\delta}_1^*\, \widetilde{\delta}_2^*\,\dots\, \widetilde{\delta}_n^*]^T
[d_{ij}][\widetilde{\delta}_1\, \widetilde{\delta}_2\dots\, \widetilde{\delta}_n]$$is then clearly a possibly unbounded positive operator on $\oplus_{i=1}^nH_\nu$. The operator $\xi\to (\sum_{1\leq i,j\leq n} d_{ij}\widetilde{\delta}_i^*\widetilde{\delta}_j)(\xi)$ can easily be derived from this one by restricting the former map to the subspace of the column space where all coordinates are equal. Using this device it is then clear that the operator $\xi\to (\sum_{1\leq i,j\leq n} d_{ij}\widetilde{\delta}_i^*\widetilde{\delta}_j)(\xi)$ must be positive definite on its domain.

\subsection{Closability criteria}

What we still need in order to effectively carry out the above strategy, is a criterion which ensures that the map $(\sum_{1\leq i,j\leq n} d_{ij}\widetilde{\delta}_i^*\widetilde{\delta}_j)$ extends to a self-adjoint positive map. In our investigation of closability we shall write $[T]$ for the closure of a closable map $T$. Extending the notion of Detailed Balance II to derivations is a good way of doing this. Specifically for a weak*-closed weak*-densely defined unital derivation $\delta$, we say that $\delta$ satisfies DBII if there exists another such derivation $\delta^\flat$, for which $\M(\delta,\delta^\flat)^{(1)}$ is weak* dense and for which we also we have that $\nu(a\delta(b))=\nu(\delta^\flat(a)b)$ for all $a, b \in \M(\delta,\delta^\flat)^{(1)}$. 

More generally we may say that the set of weak*-closed derivations $\delta_1, \dots, \delta_n$ collectively satisfy DBII if there exist derivations $\delta_i^\flat$ such that for each $i$, the pair $(\delta_i,\delta_i^\flat)$ satisfy DBII with respect to each other, and if the entire collection $\mathcal{F}_{DB}$ of 2$n$ derivations admits a weak* dense unital subalgebra ${\M}^{(2)}(\mathcal{F}_{DB})$ of ``twice differentiable'' elements. 
In order to avoid any possibility of confusion we emphasize that in the above definition of DBII for derivatives we do not demand any positivity for either $\delta$ or $\delta^\flat$. Consequently DBII for derivatives need not imply any commutativity with the modular operator. To have this property for a derivative an extra selection procedure should be employed, cf. Theorem \ref{quadder} and Proposition \ref{dermod}. The reason for identifying this property for derivatives by the name DBII, is because it is in the same vein as that given in Definition \ref{DBIIdef}.

Using the fact that $\overline{d_{ij}}=d_{ji}$ (since $D=[d_{ij}]$ is positive definite), it is then an exercise to see that for each $i$, and all $a,b\in {\M}^{(2)}$ we have that   
\begin{eqnarray*}
\langle \widetilde{\delta}_i^*(\eta(a)),\eta(b)\rangle &=& \langle \eta(a),\widetilde{\delta}_i\eta(b)\rangle\\
&=& \langle \eta(a),\eta(\delta_i(b))\rangle\\
&=& \nu(\delta_i(b)^*a)=\nu(\delta_i(b^*)a)=\nu(b^*\delta_i^\flat(a))\\
&=& \langle \eta(\delta_i^\flat(a)),\eta(b)\rangle\\
&=& \langle \widetilde{\delta}_i^\flat(\eta(a)),\eta(b)\rangle.
\end{eqnarray*}
This shows that $\widetilde{\delta}_i^*$ is an extension of $\widetilde{\delta}_i^\flat$ and therefore densely defined. 
So the restriction of $\widetilde{\delta}_i$ to $\eta({\M}^{(2)})$ is closable, and $\eta({\M}^{(2)})$ therefore a core for the closure. All of this enables one to conclude that $\widetilde{\delta}_i^{**}$ is the closure of $\widetilde{\delta}_i$, and $\widetilde{\delta}_i^*$ the closure of $\widetilde{\delta^\flat}_i$ \cite[Theorem 2.7.8]{KR}. 
Now observe that each $\delta_i$ and each $\delta_i^\flat$ maps ${\M}^{(2)}$ into ${\M}^{(1)}$. 
Having noted this fact, the above argument may now be extended to show that for all $a,b\in {\M}^{(2)}$ we have that   
\begin{eqnarray*}
\langle\sum_{1\leq i,j\leq n} d_{ij}\widetilde{\delta}_i^*\widetilde{\delta}_j(\eta(a)),\eta(b)\rangle &=& 
\sum_{1\leq i,j\leq n} \langle \eta(a),\overline{d_{ij}}\widetilde{\delta}_j^*\widetilde{\delta}_i(\eta(b))\rangle\\
&=& \sum_{1\leq i,j\leq n} \langle \eta(a),d_{ji}\widetilde{\delta}_j^*\widetilde{\delta}_i(\eta(b))\rangle\\
&=& \langle \eta(a), \sum_{1\leq i,j\leq n} d_{ij}\widetilde{\delta}_i^*\widetilde{\delta}_j(\eta(b))\rangle
\end{eqnarray*}
Thus $(\sum_{1\leq i,j\leq n} d_{ij}\widetilde{\delta}_i^*\widetilde{\delta}_j)^*$ extends $\sum_{1\leq i,j\leq n} d_{ij}\widetilde{\delta}_i^*\widetilde{\delta}_j$ and is therefore densely defined. Writing $T$ for $\sum_{1\leq i,j\leq n} d_{ij}\widetilde{\delta}_i^*\widetilde{\delta}_j$, this will as before ensure that $T$ is closable with closure $[T]=T^{**}$, and with the space $\eta({\M}^{(2)})$ being a core of this closure \cite[Theorem 2.7.8]{KR}. This then further shows that $T^{**}=[T]\subseteq T^*$and hence on taking adjoints that $T^{**}\subseteq[T]^*= T^*$. It clearly follows that $[T]=T^{**}=T^*$, thereby showing that the closure of $\sum_{1\leq i,j\leq n} d_{ij}\widetilde{\delta}_i^*\widetilde{\delta}_j$ is self-adjoint. We have already noted that $\sum_{1\leq i,j\leq n} d_{ij}\widetilde{\delta}_i^*\widetilde{\delta}_j$ is positive-definite on its natural domain. We therefore arrive at the following conclusion:

\begin{proposition}\label{Lclosed}
Suppose that the set of weak*-closed derivations $\delta_1, \dots, \delta_n$ collectively satisfy DBII. Then the operator $\sum_{1\leq i,j\leq n} d_{ij}\widetilde{\delta}_i^*\widetilde{\delta}_j$ is closable with the closure $[\sum_{1\leq i,j\leq n} d_{ij}\widetilde{\delta}_i^*\widetilde{\delta}_j]$ a self-adjoint positive definite operator for which we will by construction have that $[\sum_{1\leq i,j\leq n} d_{ij}\widetilde{\delta}_i^*\widetilde{\delta}_j](\eta(\I))=0$. If the matrix $[d_{ij}]$ is real, then $[\sum_{1\leq i,j\leq n} d_{ij}\widetilde{\delta}_i^*\widetilde{\delta}_j]$ is also reality-preserving.
\end{proposition}

There are surely examples of Laplacians in the literature which have dense domain and are not closed. To clarify this conundrum, we firstly point out that in the above setting the assumption of DBII makes all the difference in that it ensures that the operator $\sum_{1\leq i,j\leq n} d_{ij}\widetilde{\delta}_i^*\widetilde{\delta}_j$ admits a self-adjoint extension, and secondly that it is not this operator that we take to be the quantum Laplacian, but its closed extension. The following is now an easy consequence of the Lumer-Philips theorem.
 
\begin{corollary}
Under the same assumptions as in the previous proposition, the closure $\mathcal{L}=-[\sum_{1\leq i,j\leq n} d_{ij}\widetilde{\delta}_i^*\widetilde{\delta}_j]$ generates a $C_0$-semigroup $\widetilde{K}_t$ of contractions on $H_\nu$. Each element of the semigroup is a positive definite operator which preserves $h^{1/2}$.
\end{corollary}

\begin{proof} For every $\xi\in \mathrm{dom}(-[\sum_{1\leq i,j\leq n} d_{ij}\widetilde{\delta}_i^*\widetilde{\delta}_j])$, 
$\langle\cdot,\xi\rangle$ is a tangent functional corresponding to $\xi$. Since $-\langle[\sum_{1\leq i,j\leq n} d_{ij}\widetilde{\delta}_i^*\widetilde{\delta}_j]\xi,\xi\rangle\leq 0$ for each such $\xi$, the operator 
$\mathcal{L}=-[\sum_{1\leq i,j\leq n} d_{ij}\widetilde{\delta}_i^*\widetilde{\delta}_j]$ is dissipative. Equivalently 
$\|(\lambda\mathrm{id}-\mathcal{L})\xi\|_2\geq \lambda\|\xi\|_2$ for every $\lambda>0$ and every $\xi\in \mathrm{dom}(\mathcal{L})$ 
(Lumer-Philips). Since $\mathcal{L}$ is self-adjoint, this is enough to ensure that if $\lambda>0$, then $\lambda\in\rho(\mathcal{L})$ with 
$\|(\lambda\mathrm{id}-\mathcal{L})^{-1}\|\leq \lambda^{-1}$. This ensures that $\mathcal{L}$ is the generator of a semigroup $\widetilde{K}_t$ of contractions. 
The fact that each member $\widetilde{K}_t$ of the semigroup is positive definite follows from the fact that each $\widetilde{K}_t$ is the strong 
limit of $e^{t\lambda \mathcal{L}(\lambda-\mathcal{L})^{-1}}$ as $\lambda\to \infty$. Since $\mathcal{L}(h^{1/2})=0$, we have that $e^{t\lambda \mathcal{L}(\lambda-\mathcal{L})^{-1}}(h^{1/2})=h^{1/2}$ for each $t$ and each $\lambda$, and hence that $\widetilde{K}_t(h^{1/2})=h^{1/2}$ as claimed.
\end{proof}

\begin{remark}\label{Lapdim} On the basis of the analysis to date, the operator \newline $\mathcal{L}=-[\sum_{1\leq i,j\leq n} d_{ij}\widetilde{\delta}_i^*\widetilde{\delta}_j]$ described in the Proposition \ref{Lclosed}, has emerged as the best candidate for the title of ``quantum Laplacian''. In the rest of this paper we shall alway have this operator in mind when speaking of a quantum Laplacian. At first sight it may seem strange to have what is essentially an ``$n$-dimensional'' Laplacian in a context devoted to large systems, and hence we pause to comment on this point.

In many situations dynamics is controlled by some a priori given group action, and it is ultimately this group action that determines the dimensionality of the quantum Laplacian. Typically the set of derivations from which the quantum Laplacian is constructed, would be a set of infinitesimal generators of this group action. Such sets of derivations may then be used to construct quantum Laplacians which by their nature should be viewed as Laplacians of the group action. One may for example consider the von Neumann algebraic version of the Haag-Kastler local algebras where the dynamics is described by a *-automorphic action of $n$-dimensional Minkowski space on the underlying algebra. If then along the same vein we had a von Neumann algebra admitting a weak*-continuous *-automorphic action of the group $\mathbb{R}^n$, one may easily obtain an $n$-dimensional space of weak* closed weak* densely defined derivations which act as the generating set of this action in the following sense: 
For such an action the derivations $\delta_i$ would be the generators of the action of the coordinate axes on $\M$. In the space of derivations, addition is defined by strong closure, with a linear combination of the $\delta_i$'s representing the scaled generator of the *-automorphism induced by the action of $\mathbb{R}^n$ along the line through origin in the direction of the corresponding linear combination of the unit vectors $e_i$ corresponding to the coordinate axes. (The proof that the generator of the action along such a line is indeed such a linear combination, is a refinement of the description of directional derivatives typically presented in an undergraduate multi-variable Calculus course.) For a detailed analysis of such and also more general actions, the reader should refer to \cite{Bra}. We again take note that although the focus of \cite{Bra} is more on norm-closed and norm-densely defined derivations, the arguments readily carry over to the present context.

As far as applications are concerned local von Neumann algebras admitting a weak* continuous automorphic action of the group $\mathbb{R}^n$, therefore presents a class very well suited to this theory. On a closely related note, the von Neumann versions of the local algebras of Haag and Kastler presents a particularly interesting class. This class is however built around $n$-dimensional Minkowski space $\mathbb{M}_n$, and so to reflect the structure of Minkowski space, the precise formulation of the Laplacian needs to be modified to an expression of the form $\mathcal{L}=-[\widetilde{\delta}_0^*\widetilde{\delta}_0^*+\sum_{1\leq i,j\leq (n-1)} d_{ij}\widetilde{\delta}_i^*\widetilde{\delta}_j]$, where $\delta_0$ is the generator of the action of the time-coordinate of $\mathbb{M}_n$, and the space $\mathrm{span}\{\delta_i\colon 1\leq i\leq(n-1)\}$ the generating space of the spatial-coordinates.
\end{remark}

\subsection{Laplacian Markov dynamics}

Encouraging as the preceding results may be, they don't go far enough, in that we seek a Laplacian which induces Markov dynamics. For this we turn to the theory of derivations which behave well with regard to quadratic forms. Our first result is essentially a type III version of \cite[Lemma 2.1 \& Proposition 2.11]{DL}.

\begin{theorem}\label{DL2-11}
Let $A$ be a positive definite reality preserving operator on $L^2(\M)$, and let $Q(x)=\|A^{1/2}(x)\|_2^2$ be the corresponding quadratic form. Let $(T_t)$ be the semigroup of contractive self-adjoint operators induced by $-A$, and $R(\lambda)$ the resolvents $(\lambda\I+A)^{-1}$. Then the following are equivalent: 
\begin{enumerate}
\item $Q(|x|)\leq Q(x)$ for all $x\in L^2_{sa}$.
\item $(\lambda\I+A)^{-1}(x)\geq 0$ for all $x\in L^2_+$ and all $\lambda>0$.
\item $T_t(x)\geq 0$ for all $x\in L^2_+$ and all $t>0$.
\end{enumerate}  
\end{theorem}

\begin{proof} Suppose that (1) holds, let $x\in L^2_+$ and $\lambda>0$ be given, and set $y=R(\lambda)x$. By the hypothesis $y=(\lambda\I+A)^{-1}(x)$ is self-adjoint. Therefore 
$\langle x, y\rangle=tr(yx)=tr(xy)=\langle y, x\rangle$ with a similar claim holding for $\langle x,|y|\rangle$. In addition 
$\langle x,y\rangle=tr(yx)=tr(x^{1/2}yx^{1/2})\leq tr(x^{1/2}|y|x^{1/2})=tr(|y|x)=\langle x,|y|\rangle$. But then 
\begin{eqnarray*}
\|(\lambda\I+A)^{1/2}(y-|y|)\|^2 &=& \langle x, y\rangle-2\langle x, |y|\rangle+\lambda\|\,|y|\,\|^2+Q(|y|)\\
&\leq& \langle x, y\rangle-2\langle x, y\rangle+\lambda\|y\|^2+Q(y)\\
&=& 0,
\end{eqnarray*}
whence $y=|y|$. So $|R(\lambda)(x)|=R(\lambda)(x)$.

\medskip

The implication(2)$\Rightarrow$(3) follows from the fact that $T_t=\lim_{n\to\infty}\frac{n}{t}R(\frac{n}{t})^n$.

\medskip

Finally suppose that (3) holds. Let $Q_t$ be defined by $Q_t(x)=\frac{1}{t}\langle(\I-T_t)(x),x\rangle$. The fact that each $T_t$ is positive definite and contractive ensures that each $(\I-T_t)$ is positive definite, and hence that each $Q_t$ is a quadratic form. For any $x\in L^2_{sa}$, the fact that the $T_t$'s preserve positivity, further ensure that 
\begin{eqnarray*}
\langle T_t(x),x\rangle &=& \langle T_t(x_+),x_+\rangle + \langle T_t(x_-),x_-\rangle-[\langle T_t(x_-),x_+\rangle + 
\langle T_t(x_+),x_-\rangle]\\
&=& \langle T_t(x_+),x_+\rangle + \langle T_t(x_-),x_-\rangle - [\langle T_t(x_-),x_+\rangle+\langle x_+, T_t(x_-)\rangle]\\
&=& \langle T_t(x_+),x_+\rangle + \langle T_t(x_-),x_-\rangle - 2\langle T_t(x_-),x_+\rangle\\
&=& \langle T_t(|x|),|x|\rangle - 4\langle T_t(x_-),x_+\rangle\\
&\leq& \langle T_t(|x|),|x|\rangle,
\end{eqnarray*}
and hence that $Q_t(|x|)\leq Q_t(x)$ for each $t$. Since $Q$ is closed and given by $Q(x)=\lim_{t\to\infty}Q_t(x)$ 
(infinities allowed), the implication follows.
\end{proof}

\begin{definition} A weak*-closed weak*-densely defined derivation on $\M$ is said to be a quadratic derivation if the prescription $\widetilde{\delta}(h^{1/4}ah^{1/4})= h^{1/4}\delta(a)h^{1/4}$ admits a closed extension to $L^2(\M)$ such that for any self-adjoint $x\in \mathrm{dom}(\widetilde{\delta})$, we will have that $|x|\in \mathrm{dom}(\widetilde{\delta})$, with $\|\widetilde{\delta}(|x|)\|\leq \|\widetilde{\delta}(x)\|$.
\end{definition}

Armed with the above technology we are now able to identify a large class of quadratic derivations. This result is a type III analogue of \cite[Proposition 5.4]{DL}. Note that \cite[Proposition 5.4]{DL} makes the slightly weaker assertion that derivations of the type considered below are so-called Dirichlet derivations.

\begin{theorem}\label{quadder}
Let $(\alpha_t)$ be a *-automorphism group on $\M$ which observes detailed balance II with respect to another automorphism group $(\alpha_t^\flat)$. Then the generators $\delta$ and $\delta^\flat$ are both quadratic derivations which are in detailed balance II with respect to each other. For the generator $\delta$, each of the prescriptions $\delta_c:(h^{c/2}ah^{(1-c)/2})\to h^{c/2}\delta(a)h^{(1-c)/2}$ (where $0\leq c\leq 1$) define a closable operator on $L^2(\M)$ which all yield the same closed operator on $L^2(\M)$. A similar claim holds for $\delta^\flat$. More generally for any $f\in \mathcal{Z}(\M)_{sa}$ both $f\delta$ and $f\delta^\flat$ are quadratic derivations. 
\end{theorem}

The proof of the theorem relies on the following very elegant lemma.

\begin{lemma}[Corollary 1.39, \cite{D}]\label{D1-39} If $A$ is the generator of a one-parameter $C_0$ semigroup $(T_t)$ on a reflexive Banach space $X$, then $\mathrm{dom}(A)=\{x\in X\colon\liminf_{t\to 0}\|T_t(x)-x\|<\infty\}$.
\end{lemma}

\begin{proof}[Proof of Theorem \ref{quadder}] 
We have by hypothesis that $\nu(\alpha_t(a) b) = \nu(a \alpha^\flat_t(b))$, for all $a,b \in \mathcal{M}$. On setting $b = \alpha_t(c)$ where $c \in \mathcal{M}$ is arbitrary, it follows that $$ \nu(ac) = \nu(\alpha_t(ac)) = \nu(\alpha_t(a)\alpha_t(c)) = \nu(a \alpha^\flat_t \alpha_t (c)).$$So one has that $\alpha^\flat_t \alpha_t (c)=c$, which in turn ensures that $\delta^\flat = - \delta$. It then follows from standard semigroup theory that $\M(\delta,\delta^\flat)^{(2)}= \mathrm{dom}(\delta^2)$ is weak* dense. For any $a,b$ in the common domain of $\delta$ and $\delta^\flat$ we have that $\nu(a\delta(b))= \lim_{t\to 0}\frac{1}{t}\nu(a(\alpha_t(b)-b)) = \lim_{t\to 0}\frac{1}{t}\nu((\alpha_t^\flat(a)-a)b) = \nu(\delta^\flat(a)b)$. The derivations $\delta$ and $\delta^\flat$ are therefore in detailed balance II with respect to each other. 

The detailed balance assumption further ensures that $(\alpha_t)$ commutes with $\sigma_t^\nu$ and that $\nu\circ\alpha_t=\nu$ for all $t$. By these facts and the fact that each $\alpha_t$ is CP, we may extend $(\alpha_t)$ to a group of *-automorphisms $(\widetilde{\alpha}_t)$ on $\M\rtimes_\nu\mathbb{R}$ by means of the prescription $\widetilde{\alpha}_t(a\lambda_s)=\alpha_t(a)\lambda_s$, for which we also have that $\widetilde{\nu}\circ\widetilde{\alpha}_t=\widetilde{\nu}$ for all $t$ \cite[Theorem 4.1]{HJX}.  
In fact by \cite[Proposition 4.5]{LM}, we then further have that $\tau\circ\widetilde{\alpha}_t=\tau$ for all $t$. This in turn ensures that $(\widetilde{\alpha}_t)$ extends to a group of continuous *-automorphisms on $\widetilde{(\M\rtimes_\nu\mathbb{R})}$ \cite[Proposition 2.73]{GLnotes}. 

With $\mathcal{A}$ denoting the von Neumann subalgebra of $\M\rtimes_\nu\mathbb{R}$ generated by the $\lambda_t$'s, we know from \cite[Theorem 4.1]{HJX} that $\widetilde{\alpha}_t(fag)=f\alpha_t(a)g$ for all $f,g\in\mathcal{A}$, all $a\in\M$ and all $t$. The continuous extension of the $\widetilde{\alpha}_t$'s to the $\tau$-measurable operators, ensures that this equality also holds for the case where $f$ and $g$ are $\tau$-measurable operators affiliated to $\mathcal{A}$. Thus for each $0\leq c\leq 1$ each $\widetilde{\alpha}_t$ then maps $h^{c/2}\M h^{(1-c)/2}$ back into itself. The space $h^{c/2}\M h^{(1-c)/2}$ is known to be dense in $L^2(\M)$. We therefore arrive at the conclusion that $L^2(\M)$ is an invariant subspace of $\widetilde{\alpha}_t$, and that each of the prescriptions $h^{c/2}a h^{(1-c)/2}\to h^{c/2}\alpha_t(a) h^{(1-c)/2}$ continuously extends to the restriction of $\widetilde{\alpha}_t$ to $L^2(\M)$. Thus the restriction of the $\widetilde{\alpha}_t$'s to $L^2(\M)$ yields a weakly (and therefore strongly) continuous group on $L^2(\M)$ for which each member maps $h^{1/2}$ back onto itself. We define $\widetilde{\delta}$ to be the generator of this group. 

Given $a\in \mathrm{dom}(\delta)$, we of course have that $\frac{1}{t}(\alpha_t(a)-a)$ is weak* convergent to $\delta(a)$ as $t\to 0$. So for any $0\leq c\leq 1$, $\frac{1}{t}h^{c/2}(\alpha_t(a)-a) h^{(1-c)/2}=\frac{1}{t}[\widetilde{\alpha}_t(h^{c/2}a h^{(1-c)/2})-h^{c/2}a h^{(1-c)/2}]$ is weakly convergent to $h^{c/2}\delta(a)h^{(1-c)/2}$. This shows that $h^{c/2}\mathrm{dom}(\delta)h^{(1-c)/2}\subset \mathrm{dom}(\widetilde{\delta})$, with $\widetilde{\delta}(h^{c/2}a h^{(1-c)/2})=h^{c/2}\delta(a)h^{(1-c)/2}$ for each $a\in \mathrm{dom}(\delta)$ and each $0\leq c\leq 1$.

We may of course repeat the above argument for the automorphism group $\alpha^\flat_t$, denoting the generator of the 
restriction of $\widetilde{\alpha}^\flat_t$ to $L^2(\M)$ by $\widetilde{\delta}^\flat$. Observe that for any $a,b\in \M$ we 
have that $tr(\widetilde{\alpha}^\flat_t(h^{1/2}a)(bh^{1/2}))= tr((h^{1/2}\alpha^\flat_t(a))(bh^{1/2}))= \nu(\alpha^\flat_t(a)b) = \nu(a\alpha_t(b))= tr((h^{1/2}a)(\alpha_t(b)h^{1/2}))= 
tr((h^{1/2}a)\widetilde{\alpha}_t(bh^{1/2}))$. Since each of $h^{1/2}\M$ and $\M h^{1/2}$ is dense in $L^2(\M)$, we therefore 
have that $tr(\widetilde{\alpha}^\flat_t(f)g) = tr(f\widetilde{\alpha}_t(g))$ for all $f,g\in L^2(\M)$. Using the easily 
verifiable fact that the $\widetilde{\alpha}^\flat_t$'s preserve adjoints on $L^2(\M)$, it follows that the group 
$(\widetilde{\alpha}^\flat_t)$ is the ``dual'' group of $(\widetilde{\alpha}_t)$, and therefore also that 
$\widetilde{\delta}^\flat = \widetilde{\delta}^*$.

For each $0\leq c\leq 1$, $h^{c/2}ah^{(1-c)/2} \to h^{c/2}\delta(a)h^{(1-c)/2}$ is a densely defined 
restriction of $\widetilde{\delta}$. Let us denote this map by $\delta_{(c)}$ for now. The map $\delta_{(c)}^\flat$ 
is defined analogously. For any $a\in\mathrm{dom}(\delta)$ and $b\in\mathrm{dom}(\delta^\flat)$ we have that 
\begin{eqnarray*}
\langle h^{c/2}\delta(a)h^{(1-c)/2},h^{c/2}b h^{(1-c)/2}\rangle &=& 
tr((h^{c/2}bh^{(1-c)/2})^*\widetilde{\delta}(h^{c/2}ah^{(1-c)/2}))\\
&=& tr(\widetilde{\delta}^\flat((h^{c/2}bh^{(1-c)/2})^*)(h^{c/2}ah^{(1-c)/2}))\\
&=& tr((\widetilde{\delta}^\flat((h^{c/2}bh^{(1-c)/2}))^*(h^{c/2}ah^{(1-c)/2}))\\
&=& tr((h^{c/2}\delta^\flat(b)h^{(1-c)/2})^*(h^{c/2}ah^{(1-c)/2}))\\
&=& \langle h^{c/2}ah^{(1-c)/2},h^{c/2}\delta^\flat(b) h^{(1-c)/2}\rangle.
\end{eqnarray*}
This shows that $\delta_{(c)}^*$ is an extension of the densely defined operator $\delta_{(c)}^\flat$. Hence the $\delta_{(c)}$ is closable with minimal closure $\delta_{(c)}^{**}$, and adjoint the minimal closure of $\delta_{(c)}^\flat$. Write $\overline{\delta_{(c)}}$ for the minimal closure of $\delta_{(c)}$. Using the 
facts proved above, we can modify that to show that for any $a\in \mathrm{dom}(\delta^\flat\delta)$, the operator 
$\delta_{(c)}^\flat\delta_{(c)}$ is a densely defined symmetric operator which appears as the restriction of the self-adjoint 
operator $\delta_{(c)}^*\overline{\delta_{(c)}}$. But the facts we proved earlier about $\widetilde{\delta}$ and 
$\widetilde{\delta}^\flat$ enable us to show that on the one hand $\widetilde{\delta}$ is another possibly larger closure of 
$\delta_{(c)}$, and on the other that $\widetilde{\delta}^*\widetilde{\delta} = 
\widetilde{\delta}^\flat\widetilde{\delta}$ is also a self-adjoint extension of the symmetric operator $\delta_{(c)}^\flat\delta_{(c)}$. The fact that $\widetilde{\delta}$ is an extension of $\overline{\delta_{(c)}}$, ensures that
$\widetilde{\delta}^*\widetilde{\delta}$ is an extension of $\delta_{(c)}^\flat\delta_{(c)}$. By the maximal 
symmetry of self-adjoint operators, these self-adjoint operators must agree. The agreement of 
$\widetilde{\delta}^*\widetilde{\delta}$ and $\delta_{(c)}^*\overline{\delta_{(c)}}$ in turn suffices to show that the closure of 
$\delta_{(c)}$ is just $\widetilde{\delta}$. We have therefore shown that for each $0\leq c\leq 1$, 
$h^{c/2}\mathrm{dom}(\delta)h^{(1-c)/2}$ is a core for $\widetilde{\delta}$.

We proceed with showing that $\|\widetilde{\alpha}_t(|x|)-|x|\|\leq \|\widetilde{\alpha}_t(x)-x\|$ for all $x\in L^2_{sa}$ in 
which case the theorem will then follow from Lemma \ref{D1-39}. Since for any $t$ and any $x\in L^2_{sa}$ we have that 
$\|\widetilde{\alpha}_t(|x|)\|^2=tr(\widetilde{\alpha}_t(|x|)^2) = 
tr(\widetilde{\alpha}_t(x)^2)=\|\widetilde{\alpha}_t(x)\|^2$, the inequality will follow if we can show that we then also 
have that $tr(x\widetilde{\alpha}_t(x))\leq tr(|x|\widetilde{\alpha}_t(|x|)$ for all $t$.

Let $x\in L^2_{sa}$ be given. We know that there exist $x_+, x_-\in L^2_+$ such that $x=x_+ -x_-$ and $|x|=x_+ + x_-$, with 
$x_+ x_-=0$. In addition we also have that the support projections $s_+$ and $s_-$ of $x_+$ and $x_-$, are mutually orthogonal 
projections in $\M$. So in this case we may take the partial isometry $u$ in the polar form $x=u|x|$ of $x$, to simply be 
$u=s_+-s_-$. We will then clearly have that $u$ commutes with $x$, and that $u^*=u$. The operators $\widetilde{\alpha}_t(u)$ 
and $\widetilde{\alpha}_t(x)$ are of course similarly related. So for 
$f=\widetilde{\alpha}_t(u)\widetilde{\alpha}_t(|x|^{1/2})|x|^{1/2}u$ and $g=\widetilde{\alpha}_t(|x|^{1/2})|x|^{1/2}$ we have that 
\begin{eqnarray*}
tr(x\widetilde{\alpha}_t(x)) &=& tr(u|x|\widetilde{\alpha}_t(u)\widetilde{\alpha}_t(|x|))\\
&=& tr(u|x|^{1/2}\widetilde{\alpha}_t(|x|^{1/2})\widetilde{\alpha}_t(u)\widetilde{\alpha}_t(|x|^{1/2})|x|^{1/2})\\
&=& tr(f^*g)\\
&\leq& tr(f^*f)^{1/2}tr(g^*g)^{1/2}\\
&=& tr(u|x|^{1/2}\widetilde{\alpha}_t(|x|)|x|^{1/2}u)^{1/2}tr(|x|^{1/2}\widetilde{\alpha}_t(|x|)|x|^{1/2})^{1/2}\\
&\leq & tr(|x|^{1/2}\widetilde{\alpha}_t(|x|)|x|^{1/2})\\
&=& tr(|x|\widetilde{\alpha}_t(|x|)).
\end{eqnarray*}

It remains to prove the final claim. So let $f\in \mathcal{Z}(\M)_{sa}$ be given. Since for any $x\in L^2(\M)$ we have that 
$\|f\widetilde{\delta}(x)\|_2 = \|\,|f|\widetilde{\delta}(x)\|_2$, we may clearly assume that $f\geq 0$. Notice that we then 
further have that $\|(f+\tfrac{1}{n}\I)\widetilde{\delta}(|x|)\|_2 \leq \|(f+\tfrac{1}{n}\I)\widetilde{\delta}(x)\|_2$ for 
all $x$ and all $n$ if and only if $\|(f+\tfrac{1}{n}\I)\widetilde{\delta}(|x|)\|_2 \leq \|(f+\tfrac{1}{n}\I)\widetilde{\delta}(x)\|_2$ for all $x$. Hence we may assume $f$ to also be invertible. For any 
$x\in L^2(\M)$ the quadratic form $Q_f(x)=\|f\widetilde{\delta}(x)\|_2^2$ then satisfies 
$Q_f(|x|)=\|f\widetilde{\delta}(|x|)\|_2^2\leq \|f\|^2\|f\widetilde{\delta}(|x|)\|_2^2\leq \|f\|^2\|f\widetilde{\delta}(x)\|_2^2\leq \|f\|^2\|f^{-1}\|^2\|f\widetilde{\delta}(x)\|_2^2 = \|f\|^2\|f^{-1}\|^2Q_f(x)$. By 
Theorem \ref{DL2-11} we then also have that $Q_f(|x|)\leq Q_f(x)$ for all $x\in L^2(\M)$, or equivalently that 
$\|f\widetilde{\delta}(|x|)\|_2\leq \|f\widetilde{\delta}(x)\|_2$. This then proves the theorem.
\end{proof}

We next show that the class of derivations described by the preceding result, induce maps maps on $L^2(\M)$ which commute with $\Delta^{1/2}$.

\begin{proposition}\label{dermod}
Let $\delta$ and $\delta^\flat$ be as in the preceding theorem. Then both $\widetilde{\delta}$ and $\widetilde{\delta}^\flat$ commute with $\Delta^{1/2}$.
\end{proposition}

\begin{proof} By the known theory of DBII \cite{MS} the maps induced by $(\alpha_t)$ on $H_\nu$ commute with $\Delta^{1/2}$. Since $\delta=\lim_{t\to 0}\frac{1}{t}(\alpha_t-\mathrm{id})$, so does the map $ah^{1/2}\to \delta(a)h^{1/2}$ ($a\in \mathrm{dom}(\delta)$). Hence the closure of this map - $\widetilde{\delta}$ also does. The proof for the map $\widetilde{\delta}^\flat$ is completely analogous. 
\end{proof}

We close this section with a description of quantum Laplacians which do induce Markov dynamics. The preceding analysis shows that this is a reasonably substantial class of Laplacians. Specifically Laplacians emanating from say a representation of $\mathbb{R}^m$ as described in Remark \ref{Lapdim} all belong to this class.

\begin{proposition}\label{LapMar}
Let $\Gamma$ and $\Gamma^\flat$ be $n$-dimensional spaces of weak*-densely defined quadratic derivations respectively spanned by $\{\delta_1, \dots,\delta_n\}$ and $\{\delta_1^\flat, \dots,\delta_n^\flat\}$ where the pairs $(\delta, \delta^\flat)$ from this space are in detailed balance II with respect to each other. (Given a linear combination $\delta$ of the $\delta_i$'s, we here wrote $\delta^\flat$ for the corresponding linear combination of the $\delta_i^\flat$'s.)  Further let $D=[d_{ij}]$ be a real positive definite invertible matrix. If $-\mathcal{L}=[\sum_{1\leq i,j \leq 1} d_{ij} \widetilde{\delta}_i^*\widetilde{\delta}_j]$ is a reality-preserving self-adjoint positive definite operator, then $\mathcal{L}$ will generate a semigroup $(\widetilde{K}_t)$ of Markov operators on $L^2(\M)$.
\end{proposition}

Situations in which one obtains an $n$-dimensional space of derivations of the type described above, are typically situations 
in which there is a weak*-continuous action of the group $\mathbb{R}^n$ on the ambient von Neumann algebra which is in 
detailed balance II with another such action. As noted in Remark \ref{Lapdim} the derivations $\delta_i$ would in this case be 
the generators of the action of the coordinate axes on $\M$, with a linear combination of the $\delta_i$'s representing the 
scaled generator of the *-automorphism induced by the action of $\mathbb{R}^n$ along the line through origin in the direction 
of the corresponding linear combination of the unit vectors $e_i$ corresponding to the coordinate axes.

\begin{proof} We know from our earlier analysis that under the given hypothesis, $\mathcal{L}$ induces a semigroup $(\mathcal{K}_t)$ of contractive positive-definite operators. Since $\mathcal{L}(h^{1/2})=0$ we also have that $\mathcal{K}_t(h^{1/2})=h^{1/2}$ for each $t$. So to see that $(\mathcal{K}_t)$ is Markov, all we still need to show is that $(\mathcal{K}_t)$ preserves positivity. 

Consider the matrix $[d_{ij}]^{1/2}=C=[c_{ij}]$. Since $D$ is a positive definite matrix over the reals, the same is true of $C$. For each fixed $i$, set $d_i=[\sum_{k=1}^nc_{ik}\delta_k]$. By hypothesis 
each $d_i$ will be a quadratic derivation. Hence $\oplus_{1\leq i\leq n}d_i$ is a quadratic derivation on  $\oplus_{1\leq i\leq n}\M$. Checking reveals that each $d_i$ is in DBII with respect to 
$d_i^\flat = [\sum_{k=1}^nc_{ik}\delta_k^\flat]$, and that the conclusions of Theorem \ref{quadder} are applicable to this 
pair. 

We remind the reader that $L^2(\oplus_{1\leq i\leq n}\M)= \oplus_{1\leq i\leq n}L^2(\M)$ (see 
\cite[Remark II.30]{Terp}). If now we apply Theorem \ref{quadder} to $\oplus_{1\leq i\leq n}d_i$, we will for any 
$(x_i)\in \mathrm{dom}(\oplus_{1\leq i\leq n}\widetilde{d}_i)$ have that 
$$\sum_{1\leq i\leq n}\|\widetilde{d}_i(|x_i|)\|_2^2\leq \sum_{1\leq i\leq n}\|\widetilde{d}_i(x_i)\|_2^2.$$
In the case where $x=x_1=\dots=x_n$, the above inequality translates to the claim that 
$$\sum_{1\leq i\leq n}\|\widetilde{d}_i(|x|)\|_2^2\leq \sum_{1\leq i\leq n}\|\widetilde{d}_i(x)\|_2^2.$$

Since $C$ is a real positive-definite matrix we will have that $\widetilde{d}_i^*=[\sum_{k=1}^nc_{ik}\widetilde{\delta}^*_k]= [\sum_{k=1}^nc_{ki}\widetilde{\delta}^*_k]$ and hence that 
\begin{eqnarray*}
\left[\sum_{i=1}^n\widetilde{d}_i^*\widetilde{d}_i\right]&=& \left[\sum_{i=1}^n[\sum_{k=1}^nc_{ki}\widetilde{\delta}^*_k][\sum_{m=1}^nc_{im}\widetilde{\delta}_m]\right]\\
&\supseteq& \left[\sum_{1\leq k,m\leq n}\sum_{i=1}^n c_{ki}c_{im}\widetilde{\delta}^*_k\widetilde{\delta}_m\right]\\ 
&=& \left[\sum_{1\leq k,m\leq n} d_{km}\widetilde{\delta}^*_k\widetilde{\delta}_m\right]\\
&=& -\mathcal{L}.
\end{eqnarray*}
Then also $\left[\sum_{i=1}^n\widetilde{d}_i^*\widetilde{d}_i\right]= \left[\sum_{i=1}^n\widetilde{d}_i^*\widetilde{d}_i\right]^* \subseteq -\mathcal{L}^*=-\mathcal{L}$, showing that in fact $-\mathcal{L} = \left[\sum_{i=1}^n\widetilde{d}_i^*\widetilde{d}_i\right]$. 
We claim that we will for any $x\in L^2(\M)$ then have that 
$$-\langle\mathcal{L}x,x\rangle = \sum_{1\leq i\leq n}\|\widetilde{d}_i(x)\|_2^2.$$To see this notice that the quadratic form on the right corresponds to the so-called ``form-sum'' of the $\widetilde{d}_i^*\widetilde{d}_i$'s and the form on the left to the strong sum. However the specific properties of the $\widetilde{d}_i$'s, ensure that these two sums agree. (See \cite[Proposition 3.26(1)]{GLnotes} and the dicussion preceding it for details. For $Q_{\mathcal{L}}(x)= -\langle\mathcal{L}x,x\rangle$ it now follows that $Q_{\mathcal{L}}(|x|)\leq Q_{\mathcal{L}}(x)$. The fact that the semigroup $(\mathcal{K}_t)$ preserves positivity, then follows from Theorem \ref{DL2-11}.
\end{proof} 

\begin{remark} In the case where $\mathcal{L}=-[\sum_{1\leq k\leq n} \widetilde{\delta}^*_k\widetilde{\delta}_k]$ we don't need the derivations $\{\delta_1, \dots,\delta_n\}$ to span a space of quadratic derivations for the proof to go through. Merely assuming $\{\delta_1, \dots,\delta_n\}$ to be a set of quadratic derivations for which $[\sum_{1\leq i \leq 1}\widetilde{\delta}_i^*\widetilde{\delta}_i]$ is a densely defined self-adjoint positive definite operator will be enough.
\end{remark}

\section{Quantum Fokker-Planck dynamics}\label{FPD}

We consider options for formulating a quantum version of equation (2.4) of \cite{AMTU}, namely 
\begin{equation}\label{FP-comm} z_t=-\mathrm{div}(\mathbf{D}\nabla z)+V(x)z,\end{equation}
where $V$ is a potential and $\mathbf{D}$ a positive definite matrix. Notice that the function $A$ in equation (2.1) of \cite{AMTU} is a real-valued function from the Sobolev space given there. This in turn ensures that the function $V$ in the above formula is also real-valued. (See equation (2.5) in \cite{AMTU}.) This in turn ensures that $\cH_c= -\mathrm{div}(\mathbf{D}\nabla)+V$ is a self-adjoint operator for which we also have that $\sigma(\cH_c)\subset[0,\infty)$.

Our preceding analysis suggests $\mathcal{L}=-[\sum_{1\leq i,j\leq n} d_{ij}\widetilde{\delta}_i^*\widetilde{\delta}_j]$ as a quantum analog of $-\mathrm{div}\mathbf{D}\nabla$. With $V$ as in section \ref{PSP}, the following then seems to be a good substitute for equation (\ref{FP-comm}):
\begin{equation}\label{eq-FPD}
f_t=-[\sum_{1\leq i,j\leq n} d_{ij}\widetilde{\delta}_i^*\widetilde{\delta}_j](f)+\widetilde{V}(f)\mbox{ for all suitable} f.
\end{equation}
(In the above $f$ is a function from $[0,\infty)$ to $\M$).

This actually accords well with the classical and quantum theory. As far as the classical theory is concerned, careful reading of pp 50-51 of \cite{AMTU}, shows that the operator $\cH_c$ induced by $z\to -\mathrm{div}(\mathbf{D}\nabla)z+V(x)z$ has ground state $z_\infty=\sqrt{\rho_\infty}$ where $\rho_\infty\in L^1$ is the unique normalized steady state of the system. Now suppose that as was the case above, the operator $\widetilde{V}$ does indeed annihilate $h^{1/2}$. Then $h^{1/2}\in L^2$ will be a ground state of the operator $\cH_q =-[\sum_{1\leq i,j\leq n} d_{ij}\widetilde{\delta}_i^*\widetilde{\delta}_j]+\widetilde{V}$ in the sense that $-[\sum_{1\leq i,j\leq n} d_{ij}\widetilde{\delta}_i^*\widetilde{\delta}_j](h^{1/2})+\widetilde{V}(h^{1/2})=0$. Also $h$ (the square of this ground state $h^{1/2}$) will then correspond to the state $\nu$. So the properties of at least this part of $\cH_q= -[\sum_{1\leq i,j\leq n} d_{ij}\widetilde{\delta}_i^*\widetilde{\delta}_j]+\widetilde{V}$ will then closely match the structure we have in \cite{AMTU}. Having identified both suitable quantum potential and Laplacian terms, we are now able to present the main result of this paper:

\begin{theorem}[Quantum Fokker-Planck dynamics]\label{FPgen} Let $\mathcal{L}= -[\sum_{1\leq i,j\leq n} d_{ij}\widetilde{\delta}_i^*\widetilde{\delta}_j]$ be a reality-preserving self-adjoint operator which generates an $L^2$ Markov semigroup, and let the operator $T$ used to define $V$ satisfy DBII. Then the following holds:
\begin{enumerate}
\item Both $\cH_q=\mathcal{L}+\widetilde{V}$ and $\cH_q^*=\mathcal{L}+\widetilde{V}^\flat$ are generators of conservative $L^2(\M)$ Markov $C_0$-semigroups $(\mathcal{T}_t)$ and $(\mathcal{T}_t^\flat)$, which are adjoints of each other.
\item The semigroups $(\mathcal{T}_t)$ and $(\mathcal{T}_t^\flat)$ are induced by conservative integrable Markov semigroups $(T_t)$ and $(T_t^\flat)$ which are KMS-adjoints of each other, in the sense that $\mathcal{T}_t=T_t^{(2)}$ and $\mathcal{T}_t^\flat=(T_t^\flat)^{(2)}$ for each $t$.
\item If $\mathcal{L}$ commutes with the modular operator $\Delta^{1/2}$, the semigroups $(T_t)$ and $(T_t^\flat)$ are in detailed balance II with respect to each other. 
\end{enumerate}
\end{theorem}

\begin{proof} Since each of $\mathcal{L}$, $\widetilde{V}$ and $\widetilde{V}^\flat$ generate contractive semigroups, each is 
dissipative. But since $\widetilde{V}$ and $\widetilde{V}^\flat=\widetilde{V}^*$ are both bounded and everywhere defined, 
this ensures that $\cH_q=\mathcal{L}+\widetilde{V}$ and $\cH_q^*=\mathcal{L}+\widetilde{V}^\flat$ are both dissipative. So by 
a well-known corollary of the Lumer-Philips theorem, they generate $C_0$-semigroups $(\mathcal{T}_t)$ and 
$(\mathcal{T}_t^\flat)$, which are adjoints of each other. Next observe that for each $\alpha>0$, the resolvents 
$\alpha(\alpha\mathrm{Id}-\mathcal{L})^{-1}$, $\alpha(\alpha\mathrm{Id}-\widetilde{V})^{-1})$ and 
$\alpha(\alpha\mathrm{Id}-\widetilde{V}^\flat)^{-1})$ each preserve order and map $h^{1/2}$ onto $h^{1/2}$. (This follows 
from the well known formulae $(\alpha\mathrm{Id}-\mathcal{L})^{-1} =\int_0^\infty e^{-\alpha t}\widetilde{K}_t\,dt$, etc.) 
Since the  Trotter-Kato theorem \cite[Corollary 3.1.31]{BR} ensures that $\mathcal{T}_t= \lim_{n\to\infty}[(\mathrm{Id}-\tfrac{t}{n}\mathcal{L})^{-1}(\mathrm{Id}-\tfrac{t}{n}\widetilde{V})^{-1}]^n$ and 
$\mathcal{T}_t^\flat = \lim_{n\to\infty}[(\mathrm{Id}-\tfrac{t}{n}\mathcal{L})^{-1}(\mathrm{Id}-\tfrac{t}{n}\widetilde{V}^\flat)^{-1}]^n$ for each 
$t>0$, it follows that each $\mathcal{T}_t$ and $\mathcal{T}_t^\flat$ also maps $h^{1/2}$ onto $h^{1/2}$ and preserves order . 
 
Claim (2) follows from the analysis in section \ref{markov}. As far as (3) is concerned recall that by Theorem 
\ref{Vcommdelta} $\widetilde{V}$ commutes with $\Delta^{1/2}$. But then each $(\lambda\mathrm{Id}-\widetilde{V})^{-1}$ ( 
$\lambda>0$) clearly also does. Similarly if $\mathcal{L}$ commutes with $\Delta^{1/2}$, then so does each 
$(\lambda\mathrm{id}-\mathcal{L})^{-1}$ ($\lambda>0$). We may then once again invoke the Trotter-Kato theorem 
\cite[Corollary 3.1.31]{BR} to conclude that each $\mathcal{T}_t$ also commutes with $\Delta^{1/2}$. The claim now follows 
from Theorem \ref{markovdb2}.
\end{proof}
 
\section{Asymptotic behaviour of the dynamics}\phantom{a}\label{asymp}

Here we pass to the special case where the operator $\widetilde{V}$ in $\cH_q= \mathcal{L}+\widetilde{V}$ is self-adjoint. This of course corresponds to the case where the Schwarz map $T$ which induces $V$ satisfies $T=T^\flat$. Notice that the fact that $|T(a)-a|^2\geq 0$ for any $a\in\M$,  ensures that $|T(a)|^2+|a|^2\geq a^*T(a)+T(a^*)a$ and hence that $2\mathrm{Re}(\nu(a^*T(a)))\leq \nu(|T(a)|^2+|a|^2)\leq \nu(T(a^*a)+a^*a)\leq 2\nu(a^*a)$. In terms of the action of $\widetilde{T}$ on $H_\nu$, this fact translates to the claim that $\mathrm{Re}(\langle \widetilde{T}(\eta(a)),\eta(a)\rangle) \leq \langle \eta(a),\eta(a)\rangle$. By continuity this inequality holds for all 
$\xi\in H_\nu$. This ensures that $-\mathrm{Re}(\langle \widetilde{V}(\xi),\xi\rangle)=\langle \xi,\xi\rangle - \mathrm{Re}(\langle \widetilde{T}(\xi),\xi\rangle)\geq 0$. Thus if $\widetilde{V}$ is self-adjoint, then $-\widetilde{V}$ is positive definite, with $-\cH_q$ then being the sum of two positive definite maps, one of which is bounded. Hence $-\cH_q$ is itself then positive definite. Given that this holds, we will use the spectral resolution $e_\lambda$ of the operator $-\cH_q= -\mathcal{L}-\widetilde{V}$ as the basis of our computations. Given the equation 
$$f_t=-\left[\sum_{1\leq i,j\leq n} d_{ij}\widetilde{\delta}_i^*\widetilde{\delta}_j\right](f)+\widetilde{V}(f)=\cH_q(f)$$
with initial condition $f(0)=f_0\in H_\nu$, the formal function $f(t)=e^{t\cH_q}f_0$ will be a solution if the formula makes 
sense, etc, etc. Using the spectral resolution of $-\cH_q$, the solution can formally be written as 
$$f(t)=\int_{\sigma(\cH_q)}e^{-t\lambda}\,de_\lambda(f_0)=(\I-s_{\cH_q})(f_0)+\int_{\sigma(\cH_q)\backslash\{0\}}e^{-t\lambda}\,de_\lambda(f_0)$$where 
$s_{\cH_q}$ is the support projection. Now if there is a ``gap'' in the spectrum of $-\cH_q$ in the sense that 
$\sigma(-\cH_q)\subset \{0\}\cup[\gamma, \infty)$ for some $\gamma>0$ then one should be able to use the fact that 
$\lambda\to e^{-t\lambda}$ is decreasing to show that 
$$\|\int_{\sigma(-\cH_q)\backslash\{0\}}e^{-t\lambda}\,de_\lambda(f_0)\|\leq e^{-t\gamma}\|\int_{\sigma(-\cH_q)\backslash\{0\}}\,de_\lambda(f_0)\|= e^{-t\gamma}\|s_{\cH_q}(f_0)\|.$$So in this case 
$f(t)\to(\I-s_{\cH_q})(f_0)$ as $t\to \infty$. However this conclusion should hold even when there is no gap. Since 
$\cH_q(h^{1/2})=0$ we know that $0\in \sigma(-\cH_q)$, and if there is no gap here, then surely $e_{[0,\gamma)}$ converges 
strongly to $e_0=(\I-s_{\cH_q})$ as $\gamma$ decreases to 0. So given $\epsilon>0$, we should be able to find a $\gamma>0$ so 
that $\|e_{[0,\gamma)}(f_0)-e_0(f_0)\|=\|e_{(0,\gamma)}(f_0)\|\leq \epsilon$. We can then argue as before to show that
$\|\int_{\sigma(-\cH_q)\cap[\gamma,\infty)}e^{-t\lambda}\,de_\lambda(f_0)\|\leq e^{-t\gamma}\|e_{[\gamma,\infty)}(f_0)\|$. 
Together these two estimates show that 
$\limsup_{t\to\infty}\|\int_{\sigma(-\cH_q)\backslash\{0\}}e^{-t\lambda}\,de_\lambda(f_0)\|\leq\epsilon$. Since $\epsilon>0$ 
was arbitrary, we have $\lim_{t\to \infty}\|\int_{\sigma(-\cH_q)\backslash\{0\}}e^{-t\lambda}\,de_\lambda(f_0)\|=0$ as 
required. So here too $f(t)\to (\I-s_{\cH_q})(f_0)$. So we get the following theorem:

\begin{theorem} Let the Laplacian $\mathcal{L}$ be the generator of an $L^2$ Markov semigroup, let $T$ satisfy DBII and let $\widetilde{V}$ be self-adjoint. For any $f_0\in H_\nu$, $T_t^{(2)}(f_0)$ will then be a solution of the equation 
$$f_t=-\left[\sum_{1\leq i,j\leq n} d_{ij}\widetilde{\delta}_i^*\widetilde{\delta}_j\right](f)+\widetilde{V}(f)=\cH_q(f)$$
with initial condition $f(0)=f_0\in H_\nu$, where $(T_t^{(2)})$ is as in the previous subsection. For this solution we have that $T_t^{(2)}(f_0)\to (\I-s_{\cH_q})(f_0)$ as $t\to \infty$.
\end{theorem} 
 
In closing we explain how this compares to the context of \cite{AMTU}. From the remark at the bottom of page 51 of \cite{AMTU} it is clear that in that setting the kernel of $\cH_c$ is the one-dimensional space $\mathrm{span}(z_\infty)$. Part of what Arnold, et al, seem to do in equation (2.4) of \cite{AMTU} is to scale the initial condition so that there $(\I-s_{\cH_c})(z_0) = z_\infty$ where they have also arranged things so that $z_\infty=\sqrt{\rho_\infty}$. In the quantum setting $\rho_\infty$ corresponds to $h=\frac{d\widetilde{\nu}}{d\tau}$ and hence $z_\infty$ to $h^{1/2}$. We do have that $h^{1/2}\in\mathrm{ker}(\cH_q)$, but we do not know if the kernel of $\cH_q$ is one-dimensional. So we are not able to follow \cite{AMTU} and arrange matters so that $(\I-s_{\cH_q})(f_0)=h^{1/2}$.

\subsection{Convergence to equilibrium}\phantom{a}\label{imply}

In this subsection we again assume that $\widetilde{V}$ is self-adjoint, and also that $\mathcal{L}$ commutes with $\Delta^{1/2}$, and generates an $L^2$ Markov semigroup. Here we will prefer the Haagerup $L^2$-approach to $H_\nu$ where $\eta(a)$ is replaced with $ah^{1/2}$. Once again consider the equation $$f_t=-\mathcal{L}(f)+\widetilde{V}(f)=-\cH_q(f)$$
but this time start with the initial condition $f(0)=ah^{1/2}\in H_\nu$ where $a\in \M_+$ is selected so that $h^{1/2}ah^{1/2}$ is a state (a norm 1 element of $L^1_+$). This means that $1=\|h^{1/2}ah^{1/2}\|_1=tr(h^{1/2}ah^{1/2})=\nu(a)$. 

We know from the analysis of the previous subsection that $\cH_q$ generates a conservative $L^2$ Markov semigroup on $L^2$ which is of the form $(T_t^{(2)})$ for some conservative Markov semigroup $(T_t)$ on $\M$, and that we also have $T_t^{(2)}(ah^{1/2})=T_t(a)h^{1/2}$.
Now recall that in the above context $T^{(2)}_t(ah^{1/2})\to (\I-s_{\cH_q})(ah^{1/2})$. 
Since $h^{1/2}T_t^{(2)}(ah^{1/2})=h^{1/2}T_t(a)h^{1/2}$ for each $t$, and since  
$T_t$'s preserve positivity, we have that $h^{1/2}T_t^{(2)}(ah^{1/2})\in L^1_+(\M)$ for each $t$. Recall that each $T_t$ satisfies DBII. We therefore have that  
$\|h^{1/2}T^{(2)}_t(ah^{1/2})\|_1=\|h^{1/2}T_t(a)h^{1/2}\|_1=tr(h^{1/2}T_t(a)h^{1/2})=\nu(T_t(a))=\nu(a)=1$. So each 
$h^{1/2}T^{(2)}_t(ah^{1/2})$ is a norm 1 state (norm 1 element of $L^1_+)$. Moreover 
by H\"older's inequality, the net $\{h^{1/2}T_t^{(2)}(ah^{1/2})\}$ must converge in $L^1$-norm to 
$h^{1/2}(\I-s_{\cH_q})(ah^{1/2})$, which must then also be a norm 1 element of $L^1_+$ since the cone $L^1_+$ is norm-closed in $L^1$. 
Although we shall not need this just yet, we pause to show that $h^{1/2}(\I-s_{\cH_q})(ah^{1/2})$ is in fact of the form 
$h^{1/2}a_0h^{1/2}$ for some positive $a_0\in \M$, where $a_0h^{1/2}$ is in the kernel of $\mathcal{H}_q$. To see this notice that we will for each $t$ have that 
$h^{1/2}T_t(a)h^{1/2}\leq \|T_t(a)\|_\infty h=\|a\|h$. So in the limit we will have that 
$0\leq h^{1/2}(\I-s_{\cH_q})(ah^{1/2})\leq \|a\|h$. Thus by Lemma \ref{RNLp}, there must exist some $a_0\in \M_+$ such 
that $h^{1/2}(\I-s_{\cH_q})(ah^{1/2})=h^{1/2}a_0h^{1/2}$, or equivalently that $(\I-s_{\cH_q})(ah^{1/2})= a_0h^{1/2}$. The above convergence in $L^2(\M)$ may now be reformulated as the claim that 
$T_t^{(2)}(ah^{1/2})\to a_0h^{1/2}$ as $t\to\infty$.

The above discussion now yields the following very nice corollary.

\begin{corollary}\label{conveq} Let $\widetilde{V}$ be self-adjoint, and let $\mathcal{L}$ commute with $\Delta^{1/2}$ and generate an $L^2$ Markov semigroup. Let $(T_t)$ be the integrable conservative Markov semigroup induced by 
$\mathcal{H}_q$ on $\M$. For any positive norm 1 element $f$ of $L^1(\M)$, each $T_t^{(1)}(f)$ will also be a positive norm 1 element, with $\{T_t^{(1)}(f)\}$ converging to some positive norm 1 element of $L^1$ as $t\to \infty$.
\end{corollary}

\begin{proof} Given $\epsilon>0$, select an element $a\in\M_+$ so that $h^{1/2}ah^{1/2}$ is a state and $\|h^{1/2}ah^{1/2} - f\|\leq \frac{\epsilon}{3}$. We know that $\{T^{(1)}_t(h^{1/2}ah^{1/2})\}$ converges to some state $h^{1/2}a_0h^{1/2}$ as $t\to\infty$. So $\{T^{(1)}_t(h^{1/2}ah^{1/2})\}$ is Cauchy, and hence we may find some $t_0$ such that $\|T^{(1)}_t(h^{1/2}ah^{1/2})-T^{(1)}_s(h^{1/2}ah^{1/2})\|_1\leq \frac{\epsilon}{3}$ whenever $t,s\geq t_0$. Now recall that the fact that $\nu\circ T_t=\nu$ for each $t$, ensures that each $T_t^{(1)}$ is a norm 1 operator on $L^1$ \cite[Theorem 5.1]{HJX}. Thus for all $s,t\geq t_0$ we have that 
\begin{eqnarray*}
&&\|T_t^{(1)}(f)-T_s^{(1)}(f)\|\\
&&\leq \|T_t^{(1)}(f)-T^{(1)}_t(h^{1/2}ah^{1/2})\|+\|T^{(1)}_t(h^{1/2}ah^{1/2})-T^{(1)}_s(h^{1/2}ah^{1/2})\|\\
&&\qquad+\|T^{(1)}_s(h^{1/2}ah^{1/2})-T_s^{(1)}(f)\|\\
&&\leq \|f-h^{1/2}ah^{1/2}\|+\|T^{(1)}_t(h^{1/2}ah^{1/2})-T^{(1)}_s(h^{1/2}ah^{1/2})\|+\|h^{1/2}ah^{1/2}-f\|\\
&&\leq \epsilon.
\end{eqnarray*}
The net $\{T^{(1)}_t(f)\}$ is therefore Cauchy, and must therefore converge to some $f_0\in L^1$. But for each $t$, $T^{(1)}_t(h^{1/2}ah^{1/2})=h^{1/2}T_t(a)h^{1/2}$ was shown to be a state in the preceding discussion. Since $\|T_t^{(1)}(f)-T^{(1)}_t(h^{1/2}ah^{1/2})\|\leq \|f-h^{1/2}ah^{1/2}\|<\epsilon$ and since $\epsilon$ was arbitrary, each $T^{(1)}_t(f)$ is therefore also a state, with limit $f_0$ then necessarily also being a state.
\end{proof}

\section{A quantum Csiszar-Kullback inequality}\label{qcki}

Corollary \ref{conveq} shows that for a specific class of generators we get a very nice theory with clear convergence to equilibrium. We now pass to the challenge of demonstrating convergence to equilibrium for a more general class of generators. To achieve this objective we shall need some additional technology in the form of a quantum Csiszar-Kullback inequality. We shall prove and refine the following result.

\begin{proposition}\label{qcs}
Let $f,g$ be positive norm 1 elements of $L^1(\M)$ and suppose that $g$ is non-singular. (So the support of $g$ is $\I$.) Then 
$\|f-g\|_1\leq 2 tr(g|(g^{-1/2}fg^{-1/2})\log(g^{-1/2}fg^{-1/2})|)$. (Here we are assuming that the action of $tr$ extends to the collection of positive operators $b$ affiliated to $\M\rtimes_\nu\mathbb{R}$ for which $\theta_s(b)=e^{-s}b$ for each $s$.) 
\end{proposition}

We note that a similar inequality holds for the Araki relative entropy. See Theorem 3.1 of \cite{HiOhTs}. However we shall ultimately need a concept of relative entropy which is sensitive to $\sigma$-strong* convergence. (See Lemma \ref{sstrong*ent}.) At this stage it is not clear to us if Araki relative entropy is sufficiently responsive to $\sigma$-strong* convergence, and for this reason we shall therefore prefer a concept of relative entropy based on the above result. 

\begin{definition} The type of relative entropy we will need to investigate here will be the one defined as follows: Let $f,g$ be positive norm 1 elements of $L^1(\M)$ and suppose that $g$ is non-singular. We define the relative entropy $S(f|g)$ to be $S(f|g)=tr(g|(g^{-1/2}fg^{-1/2})\log(g^{-1/2}fg^{-1/2})|)$.
\end{definition}

\textbf{Note:} With $f$ and $g$ as above we have that $\theta_s(f)=e^{-s}f$ and $\theta_s(g)=e^{-s}g$ for all $s$, and hence that $\theta_s(g^{-1/2}fg^{-1/2})=(g^{-1/2}fg^{-1/2})$ for all $s$. The operator $(g^{-1/2}fg^{-1/2})$ may only be affiliated to $\cM$ and not actually in $\tcM$. So for this reason this equality is not enough to ensure that $(g^{-1/2}fg^{-1/2})$ actually belongs to $L^\infty(\M)\equiv\M$, but it is enough to ensure that it is ``affiliated'' to $L^\infty(\M)\equiv\M$. So for any real-valued Borel function $F$, $F(g^{-1/2}fg^{-1/2})$ will still be affiliated to $L^\infty(\M)\equiv\M$. So for any $a\in L^1(\M)$, we have a shot at making sense of $tr(aF(g^{-1/2}fg^{-1/2}))$.

\begin{proof}[Proof of Proposition \ref{qcs}]
Let $f,g$ be as in the hypothesis. For any $s\in \mathbb{R}$ we have that $\theta_s(\chi_{(0,\infty)}(f-g)) = \chi_{(0,\infty)}(\theta_s(f-g))=\chi_{(0,\infty)}(e^{-s}(f-g))=\chi_{(0,\infty)}(f-g)$. Thus $\chi_{(0,\infty)}(f-g)\in \M$. Similarly $\chi_{(-\infty,0]}(f-g)\in \M$. 

If $f\geq g$, then $\|f-g\|_1=tr(|f-g|)=tr(f-g)=tr(f)-tr(g)=1-1=0$. So we can only have $f\geq g$ if $f=g$. The same claim clearly holds 
if $f\leq g$. In either of these cases the required inequality clearly holds. We pass to the case where $f\neq g$. Note that $|f-g|=(f-g)\chi_{(0,\infty)}(f-g)-(f-g)\chi_{(-\infty,0]}(f-g)$. Observe that $\|f-g\|_1=tr(|f-g|)= tr((f-g)\chi_{(0,\infty)}(f-g))+tr((g-f)\chi_{(-\infty,0]}(f-g))$ with 
\begin{eqnarray*}
tr((g-f)\chi_{(-\infty,0]}(f-g))&=&tr(g\chi_{(-\infty,0]}(f-g)) - tr(f\chi_{(-\infty,0]}(f-g))\\
&=&tr(g-g\chi_{(0,\infty)}(f-g)) - tr(f-f\chi_{(0,\infty)}(f-g))\\
&=&[1-tr(g\chi_{(0,\infty)}(f-g))] - [1-tr(f\chi_{(0,\infty)}(f-g))]\\
&=&tr((f-g)\chi_{(0,\infty)}(f-g)).
\end{eqnarray*}
It therefore follows that 
$\|f-g\|_1=2tr((f-g)e_0)$ where $e_0=\chi_{(0,\infty)}(f-g)$. Notice that $s-1\leq s\log(s)$ for all $s\geq 0$. We may now use this fact with the non-singularity of $g$ to see that
\begin{eqnarray*}
\|f-g\|_1&=&2tr((f-g)e_0)\\
&=&2tr((g^{-1/2}fg^{-1/2}-\I)(g^{1/2}e_0g^{1/2}))\\
&\leq& 2tr((g^{-1/2}fg^{-1/2})\log(g^{-1/2}fg^{-1/2})(g^{1/2}e_0g^{1/2}))\\
&\leq& 2tr(|(g^{-1/2}fg^{-1/2})\log(g^{-1/2}fg^{-1/2})|(g^{1/2}e_0g^{1/2}))\\
&\leq& 2tr(|(g^{-1/2}fg^{-1/2})\log(g^{-1/2}fg^{-1/2})|g).
\end{eqnarray*}
\end{proof}

In the special case where the elements $f,g\in L^2(\M)$ are of the form $f= h^{1/2}ah^{1/2}$ and $g= h^{1/2}a_0h^{1/2}$ the above proposition can be refined. To do this we start by recalling that $h^{1/2p}\M h^{1/2p}$ is dense in $L^p(\M)$ \cite[Proposition 7.52]{GLnotes}, and then use this fact to modify the earlier definition of relative entropy. 

\begin{definition} Suppose we have two positive elements $a,a_0\in \M$ where here $a_0\in\M$ is selected so that 
$a_0\geq \epsilon\I$ for some $\epsilon>0$. In this setting we may then define a slightly modified relative entropy by the prescription $$\widetilde{S}(h^{1/2}ah^{1/2}|h^{1/2}a_0h^{1/2})= \nu(a_0^{1/2}|(a_0^{-1/2}aa_0^{-1/2})\log(a_0^{-1/2}aa_0^{-1/2})|a_0^{1/2}).$$
\end{definition}

In the above setting we obtain the following modified version of the Csiszar-Kullback inequality:

\begin{proposition}\label{qcs2}
Let $h^{1/2}ah^{1/2}$ and $h^{1/2}a_0h^{1/2}$ be two positive norm 1 elements of $L^1(\M)$ for which $a_0\geq \epsilon 
\I$. Then $\|h^{1/2}ah^{1/2}-h^{1/2}a_0h^{1/2}\|_1\leq 2\widetilde{S}(h^{1/2}ah^{1/2}|h^{1/2}a_0h^{1/2})$.
\end{proposition}

\begin{proof}
The proof is a minor modification of the proof of Proposition \ref{qcs}. 
Let $f=h^{1/2}ah^{1/2}$ and $g=h^{1/2}a_0h^{1/2}$ be as in the hypothesis. For any $s\in \mathbb{R}$ we have that $\theta_s(\chi_{(0,\infty)}(f-g)) =\chi_{(0,\infty)}(\theta_s(f-g))=\chi_{(0,\infty)}(e^{-s}(f-g))=\chi_{(0,\infty)}(f-g)$. Thus $\chi_{(0,\infty)}(f-g)\in \M$. Similarly $\chi_{(-\infty,0]}(f-g)\in \M$. 

If $f\geq g$, then $\|f-g\|_1=tr(|f-g|)=tr(f-g)=tr(f)-tr(g)=1-1=0$. So we can only have $f\geq g$ if $f=g$. The same claim clearly holds 
if $f\leq g$. In either of these cases the required inequality clearly holds. We pass to the case where $f\neq g$. Note that $|f-g|=(f-g)\chi_{(0,\infty)}(f-g)-(f-g)\chi_{(-\infty,0]}(f-g)$. Observe that $\|f-g\|_1=tr(|f-g|)= tr((f-g)\chi_{(0,\infty)}(f-g))+tr((g-f)\chi_{(-\infty,0]}(f-g))$ with 
\begin{eqnarray*}
tr((g-f)\chi_{(-\infty,0]}(f-g))&=&tr(g\chi_{(-\infty,0]}(f-g)) - tr(f\chi_{(-\infty,0]}(f-g))\\
&=&tr(g-g\chi_{(0,\infty)}(f-g)) - tr(f-f\chi_{(0,\infty)}(f-g))\\
&=&[1-tr(g\chi_{(0,\infty)}(f-g))] - [1-tr(f\chi_{(0,\infty)}(f-g))]\\
&=&tr((f-g)\chi_{(0,\infty)}(f-g)).
\end{eqnarray*}
It therefore follows that 
$\|f-g\|_1=2tr((f-g)e_0)$ where $e_0=\chi_{(0,\infty)}(f-g)$. Notice that $s-1\leq s\log(s)$ for all $s\geq 0$. We may use this fact to see that
\begin{eqnarray*}
\|f-g\|_1&=&2tr((f-g)e_0)\\
&=& 2tr((h^{1/2}(a-a_0)h^{1/2}e_0)\\
&=& 2tr((h^{1/2}a_0^{1/2}(a_0^{-1/2}aa_0^{-1/2}-\I)a_0^{1/2}h^{1/2}e_0)\\
&\leq& 2tr((a_0^{-1/2}aa_0^{-1/2})\log(a_0^{-1/2}aa_0^{-1/2})(a_0^{1/2}h^{1/2}e_0h^{1/2}a_0^{1/2}))\\
&\leq& 2tr(|(a_0^{-1/2}aa_0^{-1/2})\log(a_0^{-1/2}aa_0^{-1/2})|(a_0^{1/2}h^{1/2}e_0h^{1/2}a_0^{1/2}))\\
&\leq& 2tr(|(a_0^{-1/2}aa_0^{-1/2})\log(a_0^{-1/2}aa_0^{-1/2})|(a_0^{1/2}ha_0^{1/2}))\\
&=& 2tr(h^{1/2}a_0^{1/2}|(a_0^{-1/2}aa_0^{-1/2})\log(a_0^{-1/2}aa_0^{-1/2})|a_0^{1/2}h^{1/2})\\
&=& 2\nu(a_0^{1/2}|(a_0^{-1/2}aa_0^{-1/2})\log(a_0^{-1/2}aa_0^{-1/2})|a_0^{1/2}).
\end{eqnarray*}
\end{proof}

\section{Convergence to equilibrium for more general settings}\label{genconv}

We remind the reader that we are now concerned with demonstrating convergence to equilibrium for a more general class of generators. We shall see that some specific assumptions regarding the behaviour of the infinitesimal generator, give us access to the quantum Csiszar-Kullback inequality which we may then use to show that in such a case some form of convergence to equilibrium also pertains.

To get some intuition we consider a dissipative thermodynamical system consisting of a single harmonic oscillator in the infinite chain of one  dimensional harmonic oscillators, for details see \cite{D2, emch}.
\begin{example}
Let $\{W(f); f \in \mathcal{H} \}$  be the set of Weyl operators generating CCR algebra, where $\mathcal{H} = L^2(R)$, cf Section 5.2 in \cite{BR2}. In particular, $W(f)W(g) = e^{- \frac{i}{2} im (f,g)} W(f+g)$, $W(f)^* = W(-f)$, $W(0) = \mathbf{1}$, $ R \ni \lambda \mapsto W(\lambda f)$ is weakly continuos.  We note that there is a distinguished $f_0 \in L^2(R)$ cf \cite{emch}. Denote by $\mathfrak{N}$ ($\mathfrak{M}$) the von Neumann algebra generated by $\{ W(zf_0), z \in C \}$ ($\{ W(g); g \perp f_0 \}$ respectively). $\mathfrak{N}$ ($\mathfrak{M}$) is associated with the distinguished harmonic oscillator (with the rest of the infinite chain of oscillators). The following strongly continuous completely positive semigroup 
\begin{equation}
T_t\left( W(zf_0)\right) = W(e^{-\kappa t}zf_0) exp\{- \Theta|z|^2(1 - e^{- 2\kappa t})/4 \}
\end{equation}
where $\kappa$ and $\Theta$ are appropriately chosen positive constants, is describing the quantum diffusion process of the distinguished  oscillator. Note, that $T_t(\left( W(zf_0)\right)  \to constant \mathbf{1}$ as $t$ goes to infinity.
\end{example}

Let $(T_t)$ be a general integrable Markov semigroup on $\M$ (so not necessarily the same as the one in section \ref{FPD}). The next ingredient we will need is to show that under suitable restrictions, $t\to\widetilde{S}(h^{1/2}T_t(a)h^{1/2}|h^{1/2}a_0h^{1/2})$ will be continuous. (Here $a$ and $a_0$ are as before.) 

\begin{remark} Let $(T_t)$ be an adjoint preserving contractive semigroup on a von Neumann algebra $\M$. It is worth observing that continuity of the map $t\to T_t(a)$ ($a\in \M$) in any of the strong, strong*, $\sigma$-strong or $\sigma$-strong* topologies, is equivalent to continuity in any of the other. For example since any $a\in \M$ may be written as the sum of two self-adjoint elements, and since $(T_t)$ preserves adjoints, the claim about the equivalence of strong and strong* continuity follows from the fact that this equivalence holds for selfadjoint choices of $a$. The equivalence of say strong and $\sigma$-strong convergence follows from the facts that these two topologies agree on the unit ball of $\M$ and that $(T_t(a))$ is uniformly bounded for any $a\in \M$.
\end{remark}

We next prove that the relative entropy functional $\widetilde{S}$ will in certain cases be continuous. This mirrors \cite[Lemma 2.8]{AMTU}.

\begin{lemma}\label{sstrong*ent} 
For any two states of the form $h^{1/2}ah^{1/2}$ and $h^{1/2}a_0h^{1/2}$ where $a_0\geq \epsilon\I$ for some $\epsilon$, the 
function $t\to\widetilde{S}(T_t^{(1)}(h^{1/2}ah^{1/2})|h^{1/2}a_0h^{1/2})$ will be continuous whenever $t\to T_t(a)$ is 
$\sigma$-strong* continuous. Note that for states of the above form, we have that 
$$\widetilde{S}(h^{1/2}ah^{1/2}|h^{1/2}a_0h^{1/2})= \nu(a_0^{1/2}|(a_0^{-1/2}aa_0^{-1/2})\log(a_0^{-1/2}aa_0^{-1/2})|a_0^{1/2}).$$ 
In particular if $T_t(a)\to a_0$ strongly, then $\widetilde{S}(h^{1/2}T_t(a)h^{1/2}|h^{1/2}a_0h^{1/2})\to 0$.
\end{lemma}

\begin{proof}
For any $t\geq 0$ we will have that 
$T_t^{(1)}(h^{1/2}ah^{1/2})= h_\nu^{1/2}T_t(a)h_\nu^{1/2}$. So to prove the required continuity for such a state, we need to show that \newline $\nu(a_0^{1/2}|(a_0^{-1/2}T_t(a)a_0^{-1/2})\log(a_0^{-1/2}T_t(a)a_0^{-1/2})|a_0^{1/2})$ converges to 
\newline $\nu(a_0^{1/2}|(a_0^{-1/2}T_s(a)a_0^{-1/2})\log(a_0^{-1/2}T_s(a)a_0^{-1/2})|a_0^{1/2})$ as $t\to s$. This should follow if we can show that $|(a_0^{-1/2}T_t(a)a_0^{-1/2})\log(a_0^{-1/2}T_t(a)a_0^{-1/2})|$ will be $\sigma$-strong* convergent to $|(a_0^{-1/2}T_s(a)a_0^{-1/2})\log(a_0^{-1/2}T_s(a)a_0^{-1/2})|$ as $t\to s$. Note further that the contractivity of $(T_t)$ ensures that $\|T_t(a)\|\leq \|a\|$ for all $t\geq 0$. So by the Stone-Weierstrass theorem we can for any $\epsilon>0$ find a polynomial $p\in C[0,r]$ where $r\geq \|a\|.\|a_0^{-1/2}\|^2$, such that $\max_{t\in[0,r]} \big|p(t)-|t\log(t)|\big|\leq \epsilon$, in which case $$\|p(a_0^{-1/2}T_t(a)a_0^{-1/2})-(a_0^{-1/2}T_t(a)a_0^{-1/2})\log(a_0^{-1/2}T_t(a)a_0^{-1/2})\| \leq \epsilon$$for all $t\geq 0$. But since on bounded sets strong convergence is preserved under multiplication, the fact that 
$T_t(a)\to T_s(a)$ strongly as $t\to s$, ensures that we then also have that $p(a_0^{-1/2}T_t(a)a_0^{-1/2})\to p(a_0^{-1/2}T_s(a)a_0^{-1/2})$ strongly. 

For any $x\in H_\nu$ we then have that 
\begin{eqnarray*}
&&\limsup_{t\to s} \|\,|(a_0^{-1/2}T_t(a)a_0^{-1/2})\log(a_0^{-1/2}T_t(a)a_0^{-1/2})|(x)\\
&&\qquad -|(a_0^{-1/2}T_s(a)a_0^{-1/2})\log(a_0^{-1/2}T_s(a)a_0^{-1/2})|(x)\|\\
&& \leq 2\epsilon\|x\|+\lim_{t\to s} \|[p(a_0^{-1/2}T_t(a)a_0^{-1/2})](x)-[p(a_0^{-1/2}T_s(a)a_0^{-1/2})](x)\|\\
&& = 2\epsilon\|x\|.
\end{eqnarray*} 
Since $\epsilon>0$ was arbitrary, it is clear that $$\lim_{t\to s} |(a_0^{-1/2}T_t(a)a_0^{-1/2})\log(a_0^{-1/2}T_t(a)a_0^{-1/2})|(x)$$ 
$$= |(a_0^{-1/2}T_s(a)a_0^{-1/2})\log(a_0^{-1/2}T_s(a)a_0^{-1/2})|(x).$$We therefore have strong convergence. 

The final statement could be deduced from the continuity we just proved, but there is a very simple direct proof. A consideration of the tangent line at $s=1$ clearly shows that $\log(s)\leq (s-1)$ and $(s-1)\leq s\log(s)$ for all $s>0$, which in turn leads to $|s\log(s)|\leq \max(s,1).|s-1|$. Thus for any $a\in\M_+$ we will have that $|(a_0^{-1/2}T_t(a)a_0^{-1/2})\log(a_0^{-1/2}T_t(a)a_0^{-1/2})|\leq \max(1,\|a\|.\|a_0^{-1/2}\|^2).|a_0^{-1/2}T_t(a)a_0^{-1/2}-\I|$. The strong convergence of $(T_t(a))$ to $a_0$ ensures that $(a_0^{-1/2}T_t(a)a_0^{-1/2}-\I)$ is strongly convergent to 0. Hence by the above inequality, so is $(|(a_0^{-1/2}T_t(a)a_0^{-1/2})\log(a_0^{-1/2}T_t(a)a_0^{-1/2})|)$. So as required, $\nu(a_0^{1/2}|(a_0^{-1/2}T_t(a)a_0^{-1/2})\log(a_0^{-1/2}T_t(a)a_0^{-1/2})|a_0^{1/2})\to 0$.
\end{proof} 

\begin{remark} The important point to note here is the ``locality'' used in the above Lemma,  where ``locality'' is understood in the following sense: fixing $a \in \mathcal{M}$, an individual trajectory $ a \mapsto T_t(a)$ is examined. Further, at the algebraic level $\mathcal{M}$ , $T_t(a)$ is assumed to be strongly$^*$ convergent to $a_0 \geq \epsilon 1$. We note that as there is an isometric identification of $L^1(\mathcal{M})$ with the predual $\mathcal{M}_*$, $h^{1/2} a_0 h^{1/2}$ with $a_0 \geq \epsilon 1$ can be identified with a faithful normal state.

Thus, an application of the quantum Csiszar-Kullback inequality, see Proposition \ref{qcs2}, says that the state $h^{1/2} a h^{1/2}$ is {\bf{ norm convergent}} to the prescribed faithful state.

On the other hand, Section \ref{asymp} indicates that for some cases another strong convergence to equilibrium can appear. To say more on such cases we note that there are spectral conditions on the infinitesimal generator of a $C_0$-semigroup $T_t$ which guarantee the stability of the semigroup $T_t$, i.e. $\lim_{t \to \infty}T_t x  =0$ for all $x$, see \cite{AB} for details.  Then, an application of a version of the Jacobs-de Leeuw- Glicksberg splitting theorem to selected semigroups,  see Section 3 in \cite{Hem},  offers another criterion for convergence to equilibrium.
\end{remark}

The last thing one should note is that this approach includes the dynamics described in Section \ref{FPD}. So suppose we have a state of the form $h^{1/2}ah^{1/2}$ where here $a\in\M$ is selected so that $a\geq \epsilon\I$ for some $\epsilon>0$. Let all notation and assumptions be as in subsection \ref{imply}. Then $h^{1/2}ah^{1/2}\geq \epsilon h$, whence $T^{(1)}_t(h^{1/2}ah^{1/2})\geq \epsilon T^{(1)}_t(h)= \epsilon h^{1/2}T_t(\I)h^{1/2}= \epsilon h$ for each $t$. On taking the limit we have that $h^{1/2}a_0h^{1/2}\geq \epsilon h$ and hence that $a_0\geq \epsilon\I$. The following then holds

\begin{proposition} Let a state $h^{1/2}ah^{1/2}$ be given where $a\in\M$ satisfies $a\geq \epsilon\I$ for some $\epsilon>0$, and let $h^{1/2}a_0h^{1/2}=\lim_{t\to\infty}T^{(1)}_t(h^{1/2}ah^{1/2})$. If $(T_t)$ is as in subsection \ref{imply}, then $\widetilde{S}(h^{1/2}T_t(a)h^{1/2}|h^{1/2}a_0h^{1/2})\to 0$. 
\end{proposition}

\begin{proof} Note that for $t\geq 1$ we have that $0\leq t\log(t)\leq t(t-1)= (t-1)^2 + (t-1)$. So for $0<t<1$ we will 
have that $0\leq -\frac{1}{t}\log(t)=\frac{1}{t}\log(\frac{1}{t})\leq (\frac{1}{t}-1)^2 + (\frac{1}{t}-1)$ and hence that  
$0\leq -t\log(t)\leq (t-1)^2 + t(1-t)\leq (t-1)^2+(1-t)$. We therefore have that $0\leq |t\log(t)|\leq (t-1)^2+|t-1|$. With $(a_0^{-1/2}T_t(a)a_0^{-1/2})-\I=v_t|(a_0^{-1/2}T_t(a)a_0^{-1/2})-\I|$ denoting the polar form of $(a_0^{-1/2}T_t(a)a_0^{-1/2})-\I$, it follows that 
\begin{eqnarray*}
&&\widetilde{S}(h^{1/2}T_t(a)h^{1/2}|h^{1/2}a_0h^{1/2})\\
&&= \nu(a_0^{1/2}|(a_0^{-1/2}T_t(a)a_0^{-1/2})\log(a_0^{-1/2}T_t(a)a_0^{1/2})|a_0^{-1/2})\\
&&\leq \nu(a_0^{1/2}[((a_0^{-1/2}T_t(a)a_0^{-1/2})-\I)^2+|(a_0^{-1/2}T_t(a)a_0^{-1/2})-\I|]a_0^{1/2})\\
&&=\nu(a_0^{1/2}[((a_0^{-1/2}T_t(a)a_0^{-1/2})-\I)^2+v_t^*((a_0^{-1/2}T_t(a)a_0^{-1/2})-\I)]a_0^{1/2})\\
&&=\nu((T_t(a)-a_0)a_0^{-1}(T_t(a)-a_0)+a_0^{1/2}v_t^*a_0^{-1/2}((T_t(a)-a_0))\\
&&=\|\pi(a_0^{-1/2})\eta(T_t(a)-a_0)\|_\nu^2+\langle \eta(T_t(a)-a_0), \pi(a_0^{-1/2}v_t)\eta(a_0^{1/2})\rangle\\
&&\leq\|a_0^{-1/2}\|_\infty^2\|\eta(T_t(a)-a_0)\|_\nu^2+\|\eta(T_t(a)-a_0)\|_\nu.\|a_0^{-1/2}\|_\infty\|\eta(a_0^{1/2})\|_\nu\\
&&=\|a_0^{-1/2}\|_\infty^2\|(T^{(2)}_t(\eta(a))-\eta(a_0)\|_\nu^2+\|(T^{(2)}_t(\eta(a))-\eta(a_0)\|_\nu.\|a_0^{-1/2}\|_\infty 
\|\eta(a_0^{1/2})\|_\nu
\end{eqnarray*}
By what we saw in section \ref{asymp}, the right hand side tends to zero as $t\to \infty$, which proves the claim.
\end{proof}

\section*{Appendix: A Dirichlet form approach to Theorem \ref{FPgen}}

Shortly after completing this paper we learned of the recent work of Cipriani and Zegarlinski \cite{CZ} on Dirichlet forms. 
We show how in the case where the operator $\widetilde{V}$ in Theorem \ref{Vcommdelta} is self-adjoint, their approach can be 
used to derive our key result Theorem \ref{FPgen} using the Dirichlet form approach of \cite{CZ}. From a mathematical 
physical point of view the importance of this observation lies in the fact that it shows that the dynamics induced by 
$\cH_q$, is ultimately dynamics induced by some ``energy potential''.

Let $(\M,L^2(\M),L^2_+(\M),J)$ be the Haagerup-Terp standard form of the $\sigma$-finite von Neumann algebra $\M$ 
\cite[Theorem 7.56]{GLnotes}. In this context, $h^{1/2}$ is the cyclic and separating vector representing the faithful normal 
state $\nu$ on $\M$ in the sense that $\nu(x)=tr(h^{1/2}xh^{1/2})=\langle xh^{1/2},h^{1/2}\rangle$ for all $x\in \M$. Within 
in this framework Cipriani and Zegarlinski \cite{CZ} define a {\it Dirichlet form} with respect to the pair $(M,\nu)$ to be a 
quadratic, semicontinuous functional 
$$\mathcal{E}:L^2(\M)\to (-\infty,+\infty]$$with domain $\mathcal{F}:=\{\xi\in L^2(M): \mathcal{E}[\xi]<+\infty\}$ satisfying 
the properties 
\begin{itemize}
\item[i)] $\mathcal{F}$ is dense in $L^2(\M)$,
\item[ii)] $\mathcal{E}[J\xi]=\mathcal{E}[\xi]$ for all $\xi\in L^2(\M)$ (reality),
\item[iii)] and $\mathcal{E}[\xi\wedge\xi_0]\le\mathcal{E}[\xi]$ for all $\xi=J\xi\in L^2(M)$ (Markovianity).
\end{itemize}
As pointed out by Cipriani and Zegarlinski, Dirichlet forms $(\mathcal{E},\mathcal{F})$ of the above type are in a one-to-one correspondence with symmetric Markov semigroups $\{T_t:t\geq 0\}$ with the correspondence given by identification of the operator $-A$ as the infinitesimal generator of the semigroup,  where $(A,\mathrm{dom}(A))$ is the self-adjoint positive definite operator $(A,D(A))$ associated with $(\mathcal{E},\mathcal{F})$ by means of the formula $\mathcal{E}[\xi]=\|\sqrt{A}\xi\|^2_2$ for all $\xi\in \mathcal{F}$.

Let $\widetilde{V}$ and $\mathcal{L}$ respectively be as in Theorem \ref{Vcommdelta} and Proposition \ref{LapMar}. From these 
results we know that both $\widetilde{V}$ and $\mathcal{L}$ generate Markov semigroups. It is clear from Proposition \ref{LapMar} that $\mathcal{L}$ generates a symmetric semigroup. As far as $\widetilde{V}$ is concerned, notice that for the operator $T$ which induces $\widetilde{V}$, the fact that $|T(a)-a|^2\geq 0$ for any $a\in\M$,  ensures that $|T(a)|^2+|a|^2\geq a^*T(a)+T(a^*)a$ and hence that $2\mathrm{Re}(\nu(a^*T(a)))\leq \nu(|T(a)|^2+|a|^2)\leq \nu(T(a^*a)+a^*a)\leq 2\nu(a^*a)$. In terms of the action of $\widetilde{T}$ on $H_\nu$, this fact translates to the claim that $\mathrm{Re}(\langle \widetilde{T}(\eta(a)),\eta(a)\rangle) \leq \langle \eta(a),\eta(a)\rangle$. By continuity this inequality holds for all $\xi\in H_\nu$. This ensures that $-\mathrm{Re}(\langle \widetilde{V}(\xi),\xi\rangle)=\langle \xi,\xi\rangle - \mathrm{Re}(\langle \widetilde{T}(\xi),\xi\rangle)\geq 0$ for all $\xi$. So if $\widetilde{V}$ is self-adjoint, which will be the case if $T=T^\flat$, $-\widetilde{V}$ will be positive definite. Thus in the case where $T=T^\flat$, $\widetilde{V}$ also generates a semigroup of symmetric Markov operators.

So on assuming that $T=T^\flat$, it follows from the above correspondence that
the quadratic forms given by $$\mathcal{E}_{V}[\xi]=\|\sqrt{-\widetilde{V}}\xi\|^2_2=\langle -\widetilde{V}\xi,\xi\rangle\mbox{ for all }\xi\in \mathcal{F}_V$$and 
$$\mathcal{E}_{\mathcal{L}}[\xi]=\|\sqrt{-\mathcal{L}}\xi\|^2_2=\langle -\mathcal{L}\xi,\xi\rangle\mbox{ for all }\xi\in \mathcal{F}_{\mathcal{L}}$$are both Dirichlet forms. But since $\widetilde{V}$ (and therefore $\mathcal{E}_{V}$) is bounded and everywhere defined, the sum $\mathcal{E}_{\mathcal{L}+\widetilde{V}}[\xi]=\|\sqrt{(-\mathcal{L}-\widetilde{V})}\xi\|^2_2=\langle (-\mathcal{L}-\widetilde{V})\xi,\xi\rangle$ of these Dirichlet forms can then be shown to be yet another Dirichlet form, with domain $\mathcal{F}_{\mathcal{L}}$. (Properly the sum of operators inside the inner product is here the so-called form sum, but in the present case the operator sum and form sum agree - see \cite[Proposition 3.26]{GLnotes} and the discussion preceding it.) It therefore follows that $\cH_q=\mathcal{L}+\widetilde{V}$ is the generator of a Markov semigroup. We anticipate a theory of assymetric Dirichlet forms for general von Neumann algebras will yield the same conclusion for the case where $T\neq T^\flat$.


\begin{thebibliography}{8888}
\bibitem{Ar} H. Araki, Relative Entropy of States of Von Neumann Algebras. {\it{Publ RIMS Kyoto Univ}} \textbf{11}(1976),  809-833.

\bibitem{Araki} H. Araki, Dynamics and potentials, \textit{Lett. Math. Phys.} \textbf{82} 297-306 (2007).

\bibitem{AB} W. Arendt \& C.J.K. Batty, Tauberian theorems and stability
of one-parameter semigroups, {\it Trans. Amer. Math. Soc.}, {\bf 306},
836(1988).

\bibitem{AF} A. Arnold, F. Fagnola \& L. Neumann, Quantum Fokker-Planck models: The Lindblad and Wigner approaches, 
\textit{Quantum \ Probability \ and \ Related \ Topics}(2008), 23--48.

\bibitem{ALMS} A.  Arnold,  J.  L. L\'opez,  P.  A. Markowich, J. Soler,  An analysis of quantum Fokker-Planck models: A Wigner function approach, \textit{Rev. Mat. Iberoamericana},  \textbf{20}  (3) 771 - 814, October, 2004.
  
\bibitem{AMTU} A. Arnold , P. Markowich , G. Toscani \& A. Unterreiter, On convex Sobolev inequalities and the rate of convergence to equilibrium for Fokker-Planck type equations, \textit{Communcations in Partial Differential Equations}\textbf{26}(1\&2)(2001), 43-100.

\bibitem{BKP1} Ch. Bahn, Ch. K. Ko, Y. M. Park, Dirichlet forms and symmetric Markovian semigroups on CCR algebras with respect to quasi-free states, \textit{J. Math. Phys.} \textbf{44}, 723 (2003) DOI 10.1063/1.1532 770.

\bibitem{BKP2} Ch. Bahn, Ch. K. Ko, Y. M. Park, Dirichlet forms and symmetric Markovian semigroups on $\textbf{Z}_2$-graded von Neumann algebras, \textit{Rev. Math. Phys.} \textbf{15} 823-845 (2003).

\bibitem{BGL} D. Bakry, I. Gentil \& M. Ledoux, \textit{Analysis and geometry of Markov diffusion operators}, Grundlehren der mathematischen Wissenschaften book series , Volume 348, Springer.

\bibitem{Bra} O. Bratteli, \textit{Derivations, Dissipations and Group Actions on $C^*$-algebras}, Lecture Notes in Mathematics Vol 1229, Springer, 1986.

\bibitem{BR} O. Bratteli \& D.W. Robinson, \textit{Operator Algebras and Quantum Statistical Mechanics 1: 2nd edition}, Texts and Monographs in Physics, Springer-Verlag, 1987.

\bibitem{BR2} O. Bratteli \& D.W. Robinson, \textit{Operator Algebras and Quantum Statistical Mechanics 2: 2nd edition}, Texts and Monographs in Physics, Springer-Verlag, 1997.

\bibitem{BR3} O. Bratteli \& D.W. Robinson, Unbounded derivations of von Neumann algebras,  \textit{Ann. Inst. H. Poincar\'e} Sect. A (N.S.) 25 (1976), no. 2, 139-164.

\bibitem{Cip} F. Cipriani, Dirichlet forms on noncommutative spaces, pp. 161-276 in {\it{Quantum Potential Theory}},
eds. U. Franz, M. Sch{\"u}rmann. LNM. vol 1954, Springer Verlag 2008.

\bibitem{CS} F. Cipriani \& J.-L. Sauvageot, Derivations as square roots of Dirichlet forms, \textit{Journal of Functional Anlysis} \textbf{201} (2003), 78-120.

\bibitem{CZ} F. Cipriani \& B. Zegarlinski, KMS Dirichlet forms, coercivity and superbounded Markovian semigroups, 	arXiv:2105.06000 [math.OA].

\bibitem{DL} E.B. Davies \& JM Lindsay, Noncommutative symmetric Markov semigroups, \emph{Math Z} \textbf{210}(1992), 379-411.

\bibitem{D} E.B. Davies, \emph{One-parameter semigroups}, Academic Press, London, 1980.

\bibitem{D2} E.B. Davies, Diffusion for weakly coupled quantum oscillators, {\it Commun. Math. Phys.} {\bf 27}(1972), 309-325.

\bibitem{Dir} P. A. M. Dirac, \textit{The Principles of Quantum Mechanics - 3rd edition}, Clarendon Press, (1947).

\bibitem{E} E.G. Effros, A matrix convexity approach to some celebrated quantum inequalities, \textit{PNAS}, \textbf{106}(2009), 1006-1008.

\bibitem{EH} E.G. Effros \& F. Hansen, Non-commutative perspectives, \textit{Ann. Funct. Anal.}  \textbf{5}(2014), 74-79.

\bibitem{emch} G.G. Emch, Non-equilibrium quantum statistical mechanics, {\it Acta Physica Austriaca}, Supp. XV (1976), 79-131.

\bibitem{FK1} I.J Fujii \&  E. Kamei, Uhlmann's interpolation methods for operator means, \textit{ Math. Japon} , \textbf{34}(1989),  541-547.

\bibitem{FK2}  I.J. Fujii \&  E. Kamei, Relative operator entropy in noncommutative information theory, \textit{Math. Japon.} \textbf{34}(1989), 341-348.

\bibitem{GIS} D. Guido, T. Isola \& S. Scarlatti, Non-symmetric Dirichlet Forms on
Semifinite von Neumann Algebras, \textit{Journal of Functional Analysis} \textbf{135}(1996), 50--75.

\bibitem{GL1} S. Goldstein \& J.M. Lindsay, $KMS$-symmetric Markov semigroups, {\it Math.\ Z.}\ 
\textbf{219} (1995), 591--608.

\bibitem{GL2} S. Goldstein \& J.M. Lindsay, Markov semigroups KMS-symmetric for a weight, \textit{Math Ann} \textbf{313} 
(1999), 39--67.

\bibitem{GLnotes} S. Goldstein \& L.E.  Labuschagne, {\it Notes on noncommutative  $L^p$ and Orlicz spaces,} {\L}odz University Press, {\L}odz, 2020. ISBN 978-83-8220-385-1, e-ISBN 978-83-8220-386-8.

\bibitem{HJX} U. Haagerup, M. Junge \& Q. Xu, A reduction method for noncommutative $L^p$-spaces and applications, \textit{TAMS}, \textbf{362}(2010), 2125-2165.

\bibitem{Hem} M. Hellmich, Quantum dynamical semigroups and decoherence, {\it Advances  Math. Physics} Volume 2011, Article ID 625978.

\bibitem{KadLiu} R. V Kadison, Z. Liu, A note on derivations of Murray-von Neumann algebras, \textit{PNAS},  \textbf{111} 2087-2093 (2014).

\bibitem{KR} RV Kadison and JR Ringrose, \textit{Fundamentals of the Theory of Operator Algebras: Vol 1}, Academic Press, New York, 1983.

\bibitem{LM} L.E. Labuschagne \& W.A. Majewski, Dynamics on noncommutative Orlicz spaces, \textit{Acta Mathematica Scientia}, \textbf{40B}(5)(2020), 1249-1270.

\bibitem{Maj2} W.A. Majewski, On quantum statistical mechanics: A study guide, \textit{Advances in Mathematical Physics}, vol. 2017, Article ID 9343717.

\bibitem{MS} W.A. Majewski \& R.F. Streater, The detailed balance condition and quantum dynamical maps,
\textit{J. Phys. A.: Math. Gen.} \textbf{31} (1998), 7981-7995.

\bibitem{MZ1} W.A. Majewski \& B. Zegarlinski, On quantum stochastic dynamics I. Spin systems on a lattice, {\it{Math. Phys. Electron. J.}} {\bf{1}} paper 2 (1995), 1-37.

\bibitem{MZ2} W.A. Majewski \& B. Zegarlinski, On quantum stochastic dynamics on noncommutative $L_p$-spaces, {\it{Lett. Math. Phys.}}  {\bf{36}}(1996), 337-349.

\bibitem{HiOhTs} F. Hiai, M. Ohya \& M. Tsukada, Sufficiency, KMS condition and relative entropy in von Neumann algebras, \textit{Pacific J. Math.}\textbf{96}(1981), 99-109.

\bibitem{P1} Y.M. Park, Construction of Dirichlet forms on standard forms of von Neumann algebras, {\it{ Infinite Dim. Anal., Quantum Prob. and related topics}} {\bf{ 3}}(2000), 1-14.

\bibitem{P2} Y.M. Park, Remarks on the structure of  Dirichlet forms on standard forms of von Neumann algebras, {\it{ Infinite Dim. Anal., Quantum Prob. and related topics}} {\bf{ 8}}(2005), 179-197.

\bibitem{Ped} G.K. Pedersen, \textit{$C^*$-algebras and their automorphism groups}, Academic Press Inc, London Mathematical Society Monographs, 1979.

\bibitem{PT}  G.K. Pedersen \& M. Takesaki, The Radon-Nikodym theorem for von Neumann algebras, {\em Acta Math.} {\bf 130} (1973), 53--87.

\bibitem{Sch} L.M. Schmitt, The Radon-Nikodym theorem for $L^p$-spaces of $W^*$-algebras, \emph{Publ RIMS Kyoto Univ} \textbf{22} (1986), 1025-1034.

\bibitem{Petz}  D. Petz, Quasi-entropies for states of a von Neumann algebra, {\em Publ RIMS Kyoto Univ} {\bf 21} (1985), 787--800.

\bibitem{Tak2} M. Takesaki, \textit{Theory of Operator Algebras II}, Springer, New York, 2003.

\bibitem{Terp} M. Terp, {\it $L^p$ spaces associated with von Neumann algebras}. K{\o}benhavs Universitet, Mathematisk Institut, Rapport No 3a (1981).

\bibitem{Y} J. Yngvason, The role of type III factors in Quantum Field Theory, {\it Rep. Math. Phys.}, \textbf{55}(2005), 135-147.
\end{thebibliography}
\end{document}